\documentclass[10pt]{amsart}

\usepackage{amsmath}
\usepackage{amssymb}
\usepackage{amsfonts}
\usepackage{amscd}
\usepackage{geometry}
\usepackage[all,cmtip]{xy}
\usepackage{url}
\usepackage{amsmath, amsthm, thm-restate}

\usepackage{amsmath}
\usepackage{amssymb}
\usepackage{amsfonts}
\usepackage{geometry}
\usepackage{url}
\usepackage{setspace}
\usepackage{amsthm}
\usepackage{amsthm}
\newtheorem{theorem}{Theorem}[section]
\newtheorem{lemma}[theorem]{Lemma}
\newtheorem{proposition}[theorem]{Proposition}
\newtheorem{corollary}[theorem]{Corollary}

\theoremstyle{definition}
\newtheorem{definition}[theorem]{Definition}
\newtheorem{example}[theorem]{Example}

\newtheorem{remark}[theorem]{Remark}

\newcommand{\Hom}{\mathrm{Hom}}

\newcommand{\Ad}{\mathrm{Ad}}

\newcommand{\Res}{\mathrm{Res}}
\newcommand{\Ker}{\mathrm{Ker}}
\newcommand{\Tr}{\mathrm{Tr}}
\newcommand{\End}{\mathrm{End}}

\begin{document}

\title{Chevalley restriction theorem for vector-valued functions on           
quantum groups}

\author{Martina Balagovi\' c }
\address{Department of Mathematics,  Massachusetts Institute
of Technology, Cambridge, MA 02139, USA}
\email{martinab@math.mit.edu}

\begin{abstract}
 We generalize Chevalley's theorem about restriction $\Res: \mathbb{C} [\mathfrak{g}]^{\mathfrak{g}} \to \mathbb{C} [ \mathfrak{h}]^W$ to the case when a semisimple Lie algebra $\mathfrak{g}$ is replaced by a quantum group and the target space $\mathbb{C}$ of the polynomial maps is replaced by a finite dimensional representation $V$ of this quantum group. We prove that the restriction map $\Res:(O_{q}(G)\otimes V)^{U_{q}(\mathfrak{g})}\to O(H)\otimes V$ is injective and describe the image.
\end{abstract}

\maketitle

\section{ Introduction} Let $\mathfrak{g}$ be a finite dimensional semisimple Lie algebra over $\mathbb{C}$, $\mathfrak{h}$ its Cartan subalgebra, $W$ its Weyl group, $G$ the connected simply connected algebraic group associated to $\mathfrak{g}$, and $H$ the maximal torus of $G$ corresponding to $\mathfrak{h}$. 

We can consider the space $\mathbb{C} [\mathfrak{g}]^{\mathfrak{g}}$ of polynomial functions from $\mathfrak{g}$ to $\mathbb{C}$, invariant with respect to the coadjoint action of $\mathfrak{g}$. Such functions can be restricted to polynomial functions on $\mathfrak{h}$. The Chevalley restriction theorem (the graded version of the Harish-Chandra isomorphism) states that the restriction map $\Res: \mathbb{C} [\mathfrak{g}]^{\mathfrak{g}} \to \mathbb{C} [\mathfrak{h}]$ is injective, and that the image is $\mathbb{C} [\mathfrak{h}]^{W}$, the space of polynomial functions $\mathfrak{h}\to \mathbb{C}$ invariant under the action of $W$.

There is also a version of this isomorphism for quantum groups, see \cite{T}.

Recently Khoroshkin, Nazarov and Vinberg \cite{KNV} generalized this result to the case when the target space of polynomial maps is $V$, a finite dimensional representation of $\mathfrak{g}$. Let $E_{i}$ be the Chevalley generators of $\mathfrak{g}$ associated to positive simple roots $\alpha_{i}$. Use dot to denote the action $X.v$ of $\mathfrak{g}$ on $V$, and also the action $X.f(y)$ of  $\mathfrak{g}$ on the value of a function $f:\mathfrak{h}\to V$ (so, if we think of $f$ as an element of $\mathbb{C}[\mathfrak{h}]\otimes V$, then $E_{i}.f=(1\otimes E_{i})f$). \cite{KNV} showed:
\begin{restatable*}{theorem}{FIRST}[KNV]
 The map $\Res: (\mathbb{C}[\mathfrak{g}]\otimes V)^{G} \to \mathbb{C}[\mathfrak{h}]\otimes V$ is injective. Its image consists of those functions $f\in \mathbb{C}[\mathfrak{h}]\otimes V$ that satisfy:
 \begin{enumerate}
\item $f\in \mathbb{C}[\mathfrak{h}]\otimes V[0]$;
\item $f$ is $W$-equivariant;
\item for every simple root $\alpha_i\in \Pi$ and every $n\in \mathbb{N}$, the polynomial $E_{i}^n .f $ is divisible by $\alpha_i^n$. 
\end{enumerate}
\label{KNV1}
\end{restatable*}

We want to generalize this description of the image of the restriction map to quantum groups. This is more convenient to do in the setting of algebraic groups, to which \cite{KNV} theorem generalizes naturally and with an almost identical proof.

Consider the Hopf algebra $O(G)$ of polynomial functions on the group $G$. It comes with a natural restriction map to $O(H)$, which can be extended to the case when the target space is $V$, giving $O(G)\otimes V\to O(H)\otimes V$. We are interested in what this restriction map does to equivariant functions $(O(G)\otimes V)^G$, namely the ones that satisfy
$f(gxg^{-1})=g.f(x)$ for all $x,g\in G$. The  result of \cite{KNV} modified to the setting of algebraic groups is:
\begin{restatable*}{theorem}{SECOND}\label{KNV2}
 The map $\Res: (O(G)\otimes V)^G \to O(H)\otimes V$ is injective. Its image consists of those functions $f\in O(H)\otimes V$ that satisfy:
 \begin{enumerate}
\item $f\in O(H)\otimes V[0]$;
\item $f$ is $W$-equivariant;
\item for every simple root  $\alpha_i$ and every $n\in \mathbb{N}$, the polynomial $E_{i}^n .f $ is divisible by $(1-e^{\alpha_i})^n$. 
\end{enumerate}
\label{KNV2}
\end{restatable*}

This setting is more convenient for generalization to quantum groups for the following reason. By the Peter-Weyl theorem, $O(G)\cong \bigoplus_{L} L^*\otimes L$, where the direct sum is taken over isomorphism classes of finite dimensional irreducible representations of $G$ (equivalently: over dominant integral weights $\lambda$ - in that case $L=L_{\lambda}$), and $L^*$ denotes the dual representation. In this setting $O(H)\cong \oplus_{\mu}\mathbb{C}_{\mu}^*\otimes \mathbb{C}_{\mu}$, where for any integral weight $\mu$, $\mathbb{C}_{\mu}$ is the one dimensional representation on which $H$ acts by the character $e^{\mu}$. With these isomorphisms, the restriction map $\Res: O(G)\to O(H)$ is easy to describe and corresponds to decomposing the irreducible representation $L$ into its weight spaces (see discussion after theorem \ref{KNV2} for details).

By analogy with the classical case, in the quantum case we define $O_{q}(G)=\bigoplus_{L} {^*L} \otimes L$, with ${^*L}$ being the left dual of $L$ (see Section \ref{3}), and the sum again being over all dominant integral weights $\lambda$, with $L=L_{\lambda}$ the irreducible representation of $U_{q}(\mathfrak{g})$ with highest weight $\lambda$. As all such representations have integral weights, there is again a natural restriction map to $O(H)$.  For $V$ a finite dimensional representation of $U_{q}(\mathfrak{g})$, we then consider the restriction map $$(O_{q}(G)\otimes V)^{U_{q}(\mathfrak{g})}\to O(H)\otimes V.$$ 

Let $E_{i}$ denote the standard generator of $U_{q}(\mathfrak{g})$ associated to $\alpha_i$, and let $q_{i}=q^{d_i}=q^{<\alpha_i,\alpha_i>/2}$. The main result of the paper is:

\begin{restatable*}{theorem}{MAIN}
\label{main}
The map $\Res: (O_{q}(G)\otimes V)^{U_{q}(\mathfrak{g})} \to O(H)\otimes V$ is injective. Its image consists of those functions $f\in O(H)\otimes V$ that satisfy:
 \begin{enumerate}
\item $f\in O(H)\otimes V[0]$;
\item $f$ is invariant under the (unshifted) action of the dynamical Weyl group (see Section \ref{terribledef})
\item for every simple root $\alpha_i$ and every $n\in \mathbb{N}$, the polynomial $E_{i}^n .f $ is divisible by $$(1-q_{i}^2e^{\alpha_i})(1-q_{i}^4e^{\alpha_i})\ldots (1-q_{i}^{2n}e^{\alpha_i}).$$
\end{enumerate}
\end{restatable*}

Obviously, this statement is a direct generalization of the one for $q=1$ case. Checking that restrictions to $O(H)\otimes V$ satisfy properties 1) and 3) is a direct computation; checking 2) requires more tools. The most involved part of the proof is checking that every function in $O(H)\otimes V$ that satisfies $(1)-(3)$ is a restriction of an element of $(O_{q}(G)\otimes V)^{U_{q}(\mathfrak{g})}$. The proof of the analogous statement in \cite{KNV} uses some basic geometric observations which are not available in the quantum case (these observations follow from $O(G)$ being an algebra of polynomial functions on the algebraic variety $G$). This is the reason their proof cannot be directly generalized. 

Instead, the space of invariants can be rewritten in another way, namely as
$$(O_{q}(G)\otimes V)^{U_{q}(\mathfrak{g})}=\bigoplus_{L} ({^*L} \otimes L\otimes V)^{U_{q}(\mathfrak{g})}\cong \bigoplus_{L} \Hom_{U_{q}(\mathfrak{g})}(L,L\otimes V).$$ This natural isomorphism composed with the restriction map above reformulates the problem in terms of traces of intertwining operators $L\to L\otimes V$, as it turns out that $\Phi\in  \Hom_{U_{q}(\mathfrak{g})}(L,L\otimes V)$ maps to the function on $H$ given by $x\mapsto \Tr|_{L}(\Phi\circ x)$. Such functions have been extensively studied in recent years, among others in \cite{EV1}, \cite{EV2}, and satisfy a number of remarkable symmetry properties and difference equations. Reframing the problem in terms of trace functions enables us to draw from those results to prove the above statement.

The roadmap of the paper is as follows. Section 2 describes the results of \cite{KNV} and the reformulations of them that we will generalize. In Section 3,  we give the framework and the statement of  the main Theorem \ref{main}. Section 4 includes some of the definitions and results necessary for the proof, most notably those of the dynamical Weyl group and trace functions. Section 5 contains the proof of the main theorem.

\subsection*{Acknowledgements} I am very grateful to Pavel Etingof for suggesting this problem and for his guidance. I also wish to thank the anonymous referee, whose thorough reading, comments and corrections improved the paper. The author's work was partially supported by the NSF grant DMS-0504847.

\section{The generalized Chevalley restriction theorem in the classical case}

\label{section2}
Through the paper, let $C=(a_{ij})$ be a Cartan matrix of finite type of size $r$, and $(\mathfrak{h}, \mathfrak{h}^*,\Pi, \Pi^{\vee})$ its realization. This means that $\mathfrak{h}$ is an $r$-dimensional vector space over $\mathbb{C}$ with a basis $\Pi^{\vee}=\{ h_{1},\ldots h_{r} \}$, $\mathfrak{h}^{*}$ its dual space with a basis of simple positive roots $\Pi=\{ \alpha_{1},\ldots \alpha_{r} \}$, and $\alpha_{i}(h_j)=a_{ji}$. The matrix $C$ is symmetrizable, so we let $d_i$ be the minimal positive integers that satisfy $d_i a_{ij}=d_j a_{ji}$. Define a symmetric bilinear form on $\mathfrak{h}^*$ by $\left< \alpha_{i} , \alpha_{j}  \right>=d_ia_{ij}$ and on $\mathfrak{h}$ by $\left< h_{i} , h_{j}  \right>=d_i^{-1}a_{ji}$. Both of these forms induce the same isomorphism $\mathfrak{h} \cong \mathfrak{h}^*$ by $\alpha_{i} \leftrightarrow d_i h_i$. Let $H$ be a complex torus of rank $r$, so that the Lie algebra of $H$ is $\mathfrak{h}$, and let $\exp:\mathfrak{h} \to H $ be the exponential map, such that its kernel is $\mathbb{Z}$-spanned by $2\pi i h_{j}$ (in other words, $\exp$ realizes $H$ as a quotient of $\mathfrak{h}$ by the lattice $\mathbb{Z}2\pi i \Pi^{\vee}$). We will write elements of $H$ by $x=\exp(h)=e^{h}, h \in \mathfrak{h}$, and characters on the group accordingly, meaning $e^{\alpha}:H\to \mathbb{C}$ is a character corresponding to $\alpha \in \mathfrak{h}^*$, such that $e^{\alpha}(e^h)= e^{\alpha(h)}$. Let $W$ be the Weyl group associated to this data. Let $P$ be the weight lattice (set of all $\lambda \in \mathfrak{
h}^*$ such that $\lambda(h_{i})\in \mathbb{Z}\, \, \forall i$), and $P_{+}$ the set of dominant integral weights ($\lambda \in P$ such that $\lambda(h_{i})\in \mathbb{N}_{0} \, \, \forall i$).

Let $\mathfrak{g}$ be a semisimple finite dimensional Lie algebra over $\mathbb{C}$ with a Cartan matrix $C$ and Cartan subalgebra $\mathfrak{h}$. Let $G$ be the connected simply connected complex algebraic group with Lie algebra $\mathfrak{g}$, maximal torus $H$, and $\exp : \mathfrak{g} \to G$ the exponential map that restricts to $\exp : \mathfrak{h} \to H$. For each root $\alpha$ let $\mathfrak{g}_{\alpha}$ be the appropriate root space, and for every simple root $\alpha_{i}$ let $E_{i}\in\mathfrak{g}_{\alpha_i},F_{i}\in\mathfrak{g}_{-\alpha_i}$ denote the Chevalley generators of $\mathfrak{g}$; these satisfy $[E_{i},F_{j}]=\delta_{ij}h_i$ and for every $i$ determine a copy of $\mathfrak{sl}_{2}$ in $\mathfrak{g}$.

For every $\lambda \in P$ let $M_{\lambda}$ be the Verma module with highest weight $\lambda$, generated by a distinguished highest weight vector $m_{\lambda}$. For every dominant integral $\lambda \in P_{+}$, the module $M_{\lambda}$ has an irreducible finite dimensional quotient that we call $L_{\lambda}$. Call the image of $m_{\lambda}$ in it $l_{\lambda}$. For any  finite dimensional $\mathfrak{g}$-module $V$ and any $\nu \in P$, set $V[\nu]=\{ v\in V | h.v=\nu (h)v \, \, \forall h \in \mathfrak{h} \}$, the weight space of $V$ of weight $\nu$.

Because $G$ is simply connected, representation theory of $G$ and $\mathfrak{g}$ is the same. In particular, for any finite dimensional $V$ with an action of $G$ and action of $\mathfrak{g}$ derived from it, the set of invariants is the same: $\{v\in V | g.v=v \, \, \forall g \in G \}=V^{G}=V^{\mathfrak{g}}=\{v\in V | X.v=0 \, \, \forall X \in \mathfrak{g} \}$. Because of this, in this section we will be passing from $G$ representations to $\mathfrak{g}$ representations and back without comments.

Consider the set $\mathbb{C}[\mathfrak{g}]$ of all polynomial functions on $\mathfrak{g}$. The group $G$ acts on it by the coadjoint action: for $f\in \mathbb{C}[\mathfrak{g}], X\in \mathfrak{g}, g\in G$, $(gf)(X)=f(\Ad(g^{-1})X)$. Let $V$ be any finite dimensional $G$ and $\mathfrak{g}$ representation; we will write both actions with a dot: $g.v$ and $X.v$. Consider the space $\mathbb{C}[\mathfrak{g}]\otimes V$ of polynomial functions on $\mathfrak{g}$ with values in $V$. Let $G$ act on this space diagonally on both tensor factors. This means that $g\in G$ maps $f\in \mathbb{C}[\mathfrak{g}]\otimes V$ to a polynomial function on $\mathfrak{g}$ given by $X\mapsto g.f(\Ad(g^{-1})X)$. Call the set of invariants with respect to this diagonal action $(\mathbb{C}[\mathfrak{g}]\otimes V)^G$; these are functions $f$ that satisfy $g.f(X)=f(\Ad(g)X)$ for all $g\in G,X\in \mathfrak{g}$.

There is an obvious restriction map $\Res: \mathbb{C}[\mathfrak{g}]\otimes V \to \mathbb{C}[\mathfrak{h}]\otimes V$, and $\Res: (\mathbb{C}[\mathfrak{g}]\otimes V)^{G} \to \mathbb{C}[\mathfrak{h}]\otimes V$. The graded version of the main result of \cite{KNV} (Theorem 2) describes this latter map.
\FIRST

We recall the proof of \cite{KNV}. We first need a technical lemma about algebraic geometry.

\begin{lemma} \label{mala}
Let $X$ be a smooth connected complex algebraic variety, $f$ a rational function on $X$, and  $Z$ is a divisor in $X$ such that $f$ is regular on $X\setminus Z$. Assume that for a generic point $z$ of $Z$ there exists a regular map $c_{z}: \bold{D}\to X$ from a formal disk $\bold{D}$ to $X$, such that 
$c_z(0)=z$, $c_{z}$ does not factor through $Z$ (so the limit $\lim_{t\to 0 } f(c_{z}(t))$ is well defined), and this limit is finite (equivalently, $f(c_{z}(t))\in \mathbb{C}[[t]]$). Then $f$ is regular at a generic point of $Z$, and hence it is a regular function on $X$.
\end{lemma}

\begin{proof}
The singular set of the rational function $f$ is a finite union of irreducible divisors, and it is by assumption contained in the divisor $Z$. So, it is enough to show that $f$ is regular at a generic point of $Z$ to see that $f$ is regular on $X$. 

This is a local problem. By localizing to an open subset of $X$, we may assume without loss of generality that $X$ is affine and $Z$ is irreducible. One may also assume that $Z$ is given by a polynomial equation $\{ Q=0\}$, for some regular function $Q$ on $X$ such that $dQ\ne 0$ for a generic point of $Z$.


Since $f$ is a rational function, there exists the smallest integer $m\ge 0$ such that the function $P=fQ^{m}$ is regular at a generic point of $Z$.

For a generic point $z\in Z$, using that $\lim_{t\to 0}f(c_{z}(t))$ is finite, we get  $$P(z)=\lim_{t\to 0}P(c_{z}(t))=\lim_{t\to 0}f(c_{z}(t))\cdot Q^m(c_{z}(t))=$$
$$=\lim_{t\to 0}f(c_{z}(t))\cdot \lim_{t\to 0}Q^m(c_{z}(t))=\lim_{t\to 0}f(c_{z}(t))\cdot 0^m.$$

If $m>0$, then this implies that $P(z)=0$ for a generic point $z$, so $P/Q$ is regular on $Z$ and we can replace $P$ by $P/Q$ and $m$ by $m-1\ge 0$, contrary to our choice of $m$ as minimal. So, $m=0$ and $f=P$ is regular at a generic point of $Z$.

\end{proof}

\begin{remark}
Note that the existence of $c_{z}$ for only one specific point $z$ does not guarantee that $f$ is regular at it. To prove that $f$ is regular at one point $z$, one would need to show that the limit is finite when approaching $z$ from any direction, not just along $c_{z}$. However, the assumption of the lemma is that a function $c_{z}$ exists for many points of $Z$ at once. In that case, as we showed, $f$ is regular at all points of $Z$, and hence the limit of $f$ is indeed finite when approaching any point of $Z$ from any direction. 

\end{remark}

\begin{remark}
We will first apply this lemma in the proof of theorem \ref{KNV1} for $X=\mathfrak{g}$, where $c_{z}(t)$ can be chosen to be linear functions, and then in the proof of \ref{KNV2} for $X=G$ an algebraic group, where $c_{z}(t)$ can be chosen to be multiplication by an appropriate element of $G$.
\end{remark}

We now recall the proof of Theorem \ref{KNV1}.

\begin{proof} 

Let us first show that the conditions 1)-3) are necessary. Let $f\in (\mathbb{C}[\mathfrak{g}]\otimes V)^{G}$, and let us abuse notation and write $f$ for $\Res f$.

\emph{ 1) is necessary:} For any $x\in H$, $h\in \mathfrak{h}$, we have 
$$x.f(h)=f(\Ad(x)h)=f(h).$$
From this is follows that $f(h)\in V[0].$

\emph{ 2) is necessary:} For $N_{G}(H)$ the normalizer  of $H$ in $G$ and $Z_{G}(H)$ the centralizer of $H$ in $G$, $Z_{G}(H)=H$, we have $W=N_G(H)/Z_G(H)=N_G(H)/H$. The space $V[0]$ is the $e^{0}=1$-eigenspace of $H$, so $N_{G}(H)$ preserves it and $H$ fixes it pointwise; therefore $W$ acts on it. Because $N_{G}(H)\subseteq G$, the functions we get are $W$- equivariant, meaning: $$f(wh)=w.f(h).$$

\emph{ 3) is necessary:} For $f$ a polynomial function on $\mathfrak{g}$ and $X,Y\in \mathfrak{g}$, let $$\frac{\partial f}{\partial X}(Y)=\frac{\mathrm d}{\mathrm d t} f(Y+tX) \vert _{t=0}   $$ be the directional derivative. We have the usual Taylor series expansion for the function of one complex variable $t\mapsto f(Y+tX)$ given by $$f(Y+tX)=\sum_{n\ge 0} \frac{1}{n!}t^n \frac{\partial ^n f}{\partial X^n}(Y).$$

Let us write down the invariance condition of $f$ with respect to $\mathrm{exp}(tE_{i})\in G$. For $h\in \mathfrak{h}$, and $t\in \mathbb{C}$, we have:
$$ \mathrm{exp}(tE_{i}).f (h)=f ( \mathrm{Ad}(\mathrm{exp}(tE_{i})) h) $$ 
Expand both sides into a power series in $t$ to get
\begin{eqnarray*} \sum_{n\ge 0} \frac{1}{n!}t^n E_{i}^n.f (h) & = & f ( \mathrm{exp}(\mathrm{ad}(tE_{i})) h) \\
& = &f ( h- t \alpha_{i}(h)E_{i}) \\
& = & \sum_{n\ge 0} \frac{1}{n!}(-1)^n t^n\alpha_i(h)^n \frac{\partial ^n  f }{\partial E_{i}^n}(h) \\
\end{eqnarray*}
Looking at the corresponding terms in the power series, we get $$E_{i}^n.f(h) =(-1)^n\alpha_i(h)^n \frac{\partial ^n  f }{\partial E_{i}^n}(h),$$ which is divisible by $\alpha_{i}^n$.

\emph{The map $\Res$ is injective:} The set of elements in $\mathfrak{g}$ that are $\Ad(G)$-conjugate to an element of $\mathfrak{h}$ is dense in $\mathfrak{g}$. So, if two $G$-equivariant polynomial functions on $\mathfrak{g}$ match on $\mathfrak{h}$, they match on its dense $G$-orbit, so they are the same.

\emph{ 1)-3) are sufficient:} So far we have seen that the image of $\Res$ is contained in the set of all functions satisfying $1)-3)$. To see that all functions satisfying 1)-3) are restrictions of equivariant functions on $\mathfrak{g}$, let $f$ be a polynomial function on $\mathfrak{h}$ satisfying 1)-3) and let us try extending it to $\mathfrak{g}$. 

Call elements of $\mathfrak{h}$ that are not fixed by any nontrivial element of the Weyl group \emph{regular}, and call the set of all such elements $\mathfrak{h}_{reg}$. It is a complement of finitely many hyperplanes of the form $\Ker \alpha$ in $\mathfrak{h}$, for $\alpha$ a root. Call elements of $\mathfrak{g}$ that are $\Ad(G)$-conjugate to an element of $\mathfrak{h}_{reg}$ \emph{regular semisimple}, and the set of all such elements $\mathfrak{g}_{rs}$. This is an algebraic variety, open and dense in $\mathfrak{g}$. More precisely, there exists a polynomial in $ \mathbb{C}[\mathfrak{g}]^{\mathfrak{g}}$, called the discriminant, such that the set  $\mathfrak{g}_{rs}$ is the complement of its zero set. The restriction of this polynomial to $\mathfrak{h}$ is the product of all roots of $\mathfrak{g}$.

We can extend the function $f$ to elements of $\mathfrak{g}_{rs}$ by  defining $f(\Ad(g)h)=g.f(h)$. More precisely, consider the diagram 
$$ \begin{CD}
G\times \mathfrak{h}_{reg} @>\overline{f}>> V  \\
@VV a V @AAA \\
\mathfrak{g}_{rs} @>i>> \mathfrak{g} \\
\end{CD}
$$
Here, $a(g,h)=\Ad(g)h$, and $i$ is the inclusion. The map $a$ is surjective and $\mathfrak{g}_{rs}$ is dense in $\mathfrak{g}$, so $i\circ a$ is dominant. It is compatible with the map $\overline{f}(g,h)=g.f(h)$. Indeed, if $i(a(g_{1}, h_{1}))=i(a(g_{2}, h_{2}))$, then $h_{1}=\Ad (g_{1}^{-1}g_{2})h_{2}$, so $g_{1}^{-1}g_{2}$ is in the normalizer of $\mathfrak{h}$ in $G$, and therefore a representative of an element of $W$. Using that $f:\mathfrak{h}\to V$  is $W$-invariant, we get  $$\overline{f}(g_{1}, h_{1})=g_{1}.f(h_{1})=g_{1}.f(\Ad(g_{1}^{-1}g_{2})h_{2})=g_{1}g_{1}^{-1}g_{2}.f(h_{2})=\overline{f}(g_{2}, h_{2}).$$
Therefore, there is a well defined rational function $\mathfrak{g}\to V$ that makes the above diagram commute. We claim this is the required extension of $f:\mathfrak{h}\to V$ to a $G$-invariant function on $\mathfrak{g}$.

The restriction of this function to $\mathfrak{h}$ is $f$, so we abuse notation and call this rational function on $\mathfrak{g}$ by the same name $f$. By construction, it is $G$-invariant on $\mathfrak{g}_{rs}$ which is dense in $\mathfrak{g}$, so it is $G$-invariant on the maximal domain in $\mathfrak{g}$ where it is regular. To prove this is the required function in $(\mathbb{C}[\mathfrak{g}]\otimes V)^{G}$, we just need to show that it is a regular function on $\mathfrak{g}$.



As it is regular on $\mathfrak{g}_{rs}$, we need to show it is regular on the divisor $D=\mathfrak{g}\setminus \mathfrak{g_{rs}}$, and for that we will use lemma \ref{mala}. The assumptions of this lemma refer to a generic point of the divisor $D$.  The set of elements whose semisimple part is conjugate to an element of $\mathfrak{h}$ which is contained in only one hyperplane of the form $\Ker \alpha$ is Zariski dense in  $D$. More precisely, the irreducible components of  $D$ are $$D_{\alpha}=\{ \textrm{elements of } \mathfrak{g} \textrm{ whose semisimple part is conjugate to an element of } \Ker \alpha \},$$ for $\alpha$ a representative of a $W$-conjugacy class of roots. Therefore we choose all representatives $\alpha$ to be simple roots. Then the set $D_{\alpha_{i}}$ is equal to
 $$D_{\alpha_{i}}=\{ \textrm{elements whose semisimple part is conjugate to an element of } \Ker \alpha_{i}\},$$
 and contains a Zariski dense subset $$D_{\alpha_{i}}'=\Ad (G)\{ h+E_{i} | h \in \Ker \alpha_{i}, h\notin \Ker\beta \,\, \forall \textrm{ root } \beta \ne \pm \alpha_{i} \}.$$



We will check the assumptions of Lemma \ref{mala} on any element $z=\Ad(g)(h+E_{i})$ the set $D_{\alpha_{i}}'$ for any $\alpha_{i}$, $g\in G, h\in \Ker\alpha_{i}$. We construct a function $c_{z}$, such that $\lim_{t\to 0}f(c_{z}(t))$ is finite. Pick $y \in \mathfrak{h}$ such that $\alpha_i(y)=1$ and define $c_{z}(t)=\Ad(g)(h+ty+E_{i})$. Clearly $c_{z}(0)=z$. The element $h+ty+E_{i}$ is conjugate, via $\mathrm{exp}(t^{-1}E_{i})$, to $$\Ad(\mathrm{exp}(t^{-1}E_{i})) (h+ty+E_{i})=h+ty,$$ which is in $\mathfrak{h}_{reg}$ for small $t\ne 0$, so $f$ is well defined there and $c_{z}$ doesn't factor through the divisor $D$. Now calculate the limit, using that $f$ is $G-$invariant:
\begin{eqnarray*}
\lim_{t\to 0} f (c_{z}(t)) &=& \lim_{t\to 0} f (\Ad(g)(h+ty+E_{i})) \\
&=& \lim_{t\to 0} g.f (h+ty+E_{i}) \\
 & =&  \lim_{t\to 0} g.f (\Ad(\mathrm{exp}(-t^{-1}E_{i})) (h+ty)) \\
& = &  \lim_{t\to 0} g. \mathrm{exp}(-t^{-1}E_{i}). f  (h+ty) \\
& = & g. \lim_{t\to 0}  \sum_{n\ge 0}\frac{1}{n!} (-t)^{-n} E_{i}^n. f  (h+ty). \\
\end{eqnarray*}
This sum is finite because $f$ takes values in $V$, a finite dimensional representation on which $E_{i}$ is nilpotent. $h+ty$ is in $\mathfrak{h}$, and by 3) every term $E_{i}^n. f  (h+ty)$ is divisible by $\alpha_i(h+ty)^n=t^n$. So we can exchange limit and sum, and all of the summands are finite when we let $t\to 0$. 

Using lemma \ref{mala}, we conclude that $f$ is regular on $\mathfrak{h}$, as required. 

\end{proof}

As mentioned in the introduction, this theorem can be restated terms of polynomial functions on the group $G$, with an almost identical proof. Let $O(G)$ be the algebra of polynomial functions on the algebraic group $G$, and $O(H)$ the algebra of polynomial functions on the subgroup $H$. There is again the obvious restriction map (quotient of algebras) that we will call $\Res:O(G)\to O(H)$. Let $\Res$ also denote the tensor product of this map with the identity map on a representation, $\Res:O(G)\otimes V \to O(H)\otimes V$. There is also a natural $G$-action on this tensor product, by acting on the first tensor factor by dual of conjugation in the group, and on the second by a given action on $V$. The invariants are then functions that satisfy: $$g.f(x)=f(gxg^{-1})\, \, \, \forall g,x\in G,$$ and the analogous theorem is:
\SECOND
\begin{proof}
Essentially, this proof is the same as proof of Theorem \ref{KNV1}. Necessity of conditions 1) and 2) follows directly. To check condition 3), calculate for $h\in \mathfrak{h}, t\in \mathbb{C}$:
\begin{eqnarray*}  
\exp(-h) \exp (t E_{i}) \exp(h) & = & \exp (Ad(\exp(-h)) tE_{i} )  \\  
 & = & \exp (\exp (ad(-h)) tE_{i}) \\
& =  & \exp ( e^{-\alpha_i(h)} tE_{i} )
\end{eqnarray*}  

It follows that $$  \exp(tE_{i} ) \exp(h) \exp(-tE_{i} ) = \exp(h)\exp(( e^{-\alpha(h)}-1)t E_{i} )  .$$ Now
\begin{eqnarray*}  
 \sum_{n\ge 0}  \frac{1}{n!} t^n  E_{i}^n .f(\exp h) &=&\exp(tE_{i}).f(\exp(h)) \\
  &=&f( \exp(tE_{i} ) \exp(h) \exp(-tE_{i} ) ) \\
  &=&f(  \exp(h)\exp( (e^{-\alpha_i(h)}-1)t E_{i} ) ) \\
  &=&  \sum_{n\ge 0}  \frac{1}{n!} (e^{-\alpha_i(h)}-1)^n t^n R_{E_{i}}^nf(\exp(h))\\
  &=&  \sum_{n\ge 0}  \frac{1}{n!} e^{-n\alpha_i(h)}(1-e^{\alpha_i(h)})^n t^n R_{E_{i}}^nf(\exp(h))
\end{eqnarray*}  
Here $R_{E_{i}}f$ denotes the derivative of $f$ with respect to left invariant vector field $E_i$. It is a polynomial function on $G$. It follows that  $E_{i}^n .f(\exp h) = e^{-n\alpha_i(h)}(1- e^{\alpha_i(h)})^n R_{E_{i}}^nf(\exp(h))$, and it is divisible by $ (1-e^{\alpha_{i}(h)})^n$.

If the function $f\in  O(H)\otimes V$ that satisfies 1)-3) can be extended to a $G$-equivariant function on $G$, this can be done  in a unique way, because the set of elements conjugate to an element of $H$ is dense in $G$. 

To see it always extends, just as in Theorem \ref{KNV1}, we first extend it to the set of regular semisimple elements of $G$, and then use lemma \ref{mala}.

Every element of $g\in G$ has a decomposition $g=g_{s}g_{u}$, where $g_{s}$ is semisimple, $g_{u}$ is unipotent and $g_{s}g_{u}=g_{u}g_{s}$. Every semisimple element is contained in some maximal torus. All maximal tori are conjugate. A semisimple element of $G$ is called \textit{regular} if  there is only one such torus containing $g_{s}$, equal to the centralizer $Z_{G}(g_{s})$. 
The set of regular elements $G_{rs}$ of $G$ is open dense in $G$. (See \cite{Borel}).

We can extend $f:H\to V$ to a $G$-invariant polynomial function on the set of all regular semisimple elements of $G$, by using that such an element is conjugate to an element of the fixed torus $H$, and that two elements of $H$ are $G$-conjugate if and only if they are $W$-conjugate. Because the set of regular semisimple elements is open dense in $G$, we can consider $f$ to be a rational function $G\to V$, regular except maybe on the set $G\setminus G_{rs}$. We will use lemma \ref{mala} to show that it is in fact regular everywhere. 

We have $G\setminus G_{rs}=\cup_{\alpha,m} D_{\alpha,m},$ where  $\alpha$ is an arbitrary root and $m$ an arbitrary integer, and
$$D_{\alpha, m}=\{\textrm{elements whose semisimple part is conjugate to} $$ $$\textrm{some} \exp(h)\in H, h\in \mathfrak{h}, \textrm{ such that } \alpha(h)=2\pi i m\}.$$  Two such sets coincide if they are labeled by $W$- conjugate roots and integers of the same parity.  For the purposes of applying lemma \ref{mala}, we will use the following dense subset of $D_{\alpha_{i},m}$:
$$D_{\alpha_{i}, m}'=\{G-\textrm{conjugates of}  \exp(h)\exp(E_{i})\in H,$$ $$h\in \mathfrak{h} \textrm{ such that } \alpha_{i}(h)=2\pi i m, \textrm{ and } e^{\beta (h)}\ne 1 \,\, \forall \textrm{ root } \beta \ne \pm \alpha_{i} \}.$$


Let $z=g\cdot \exp(h)\exp(E_{i})\cdot g^{-1}$ be an arbitrary element of $D_{\alpha_{i}, m}'$. Pick $y \in \mathfrak{h}$ with $\alpha_{i}(y)=1$. Then $$[y,E_{i}]=E_{i} \,\,\,\,\ [h,E_{i}]=2i\pi mE_{i},$$ so for $t\in \mathbb{C}$ $$ \exp(ty)E_{i}=e^{t}E_{i}\exp(ty)  \,\,\,\, \exp(h)E_{i}=E_{i}\exp(h).$$ From this it follows
$$\exp(\frac{-1}{e^{-t}-1}E_{i}) \exp(h+ty)\exp(E_{i})\exp(\frac{1}{e^{-t}-1}E_{i})= \exp(h+ty),$$ in other words $\exp(h+ty)\exp(E_{i})$ is conjugate to a regular element of $H$. We define $$c_{z}(t)=g\cdot  \exp(h+ty)\exp(E_{i}) \cdot g^{-1}$$
and calculate, as before:
\begin{eqnarray*}  
\lim_{t\to 0} f(c_{z}(t)) & =& \lim_{t\to 0}  f(g\cdot \exp(h+ty)\exp(E_{i})\cdot g^{-1})\\
&=&  \lim_{t\to 0}  g.f(\exp(h+ty)\exp(E_{i}))\\
&=&  \lim_{t\to 0}  g.f(\exp(\frac{1}{e^{-t}-1}E_{i})\exp(h+ty)\exp(\frac{-1}{e^{-t}-1}E_{i}))\\
&=&  \lim_{t\to 0}  g. \exp(\frac{1}{e^{-t}-1}E_{i}).f(\exp(h+ty))\\
& =& g. \lim_{t\to 0} \sum_{n\ge 0} \frac{1}{n!}\frac{1}{(e^{-t}-1)^{n}} E_{i}^{n}.f(\exp(h+ty)).\\
\end{eqnarray*}  
This sum is finite, and every $E_{i}^n.f(\exp(h+ty))$ is by assumption 3) divisible by $$(1-e^{\alpha_i})^n(\exp(h+ty))=(1-e^{\alpha_i(h+ty)})^n=(1-e^{2i\pi m+t})^n=(1-e^{t})^n.$$ Since $\lim_{t\to 0} (\frac{1-e^{t}}{e^{-t}-1})^n=\lim_{t\to 0} e^{nt}=1$, we see that  the limit of every summand is finite. So, $\lim_{t\to 0} f(c_{z}(t))$ is finite, and by lemma \ref{mala}, $f$ is regular at the generic point of $G\setminus G_{rs}$, and hence it is regular everywhere.
\end{proof}

The main reason for reformulating Theorem \ref{KNV1} in terms of Theorem \ref{KNV2} is that the latter allows generalization to quantum groups. Namely, use the Peter-Weyl theorem to write $$O(G)\cong \bigoplus_{L\in \widehat{G}} L^*\otimes L.$$ Here the sum is over $\widehat{G}$, the set of irreducible finite dimensional representations $L$ of $G$; equivalently, it is over all dominant integral weights $\mu \in P_{+}$, with $L=L_{\mu}$. The module $L^*$ is the dual space of $L$, with the natural $G$ action $g\varphi=\varphi\circ g^{-1}$. The isomorphism $A:\bigoplus_{L\in \widehat{G}} L^*\otimes L\to O(G)$ is determined by sending $\varphi\otimes l \in L^*\otimes L$ to a function on $G$ given by $x\mapsto A(\varphi\otimes l)(x)=\varphi(xl)$. It is a matrix coefficient of $L$, and therefore a polynomial function on $G$. If we put the natural action of $G$ on every tensor product $L^*\otimes L$, meaning letting $g\in G$ act by $g\otimes g$, then $A$ is an isomorphism of $G$ representations: $(Ag(\varphi\otimes l))(x)=(A(\varphi\circ g^{-1})\otimes (gl))(x)= \varphi(g^{-1}xgl) =A(\varphi\otimes l)(gxg^{-1})$. 

The action we had on $O(G)\otimes V$ was the natural action on the tensor product, so $A\otimes id_{V}:\bigoplus_{L\in \widehat{G}} L^*\otimes L\otimes V  \to O(G)\otimes V$ is also an  isomorphism of representations. There is also a natural isomorphism $B:L^*\otimes L\otimes V \to \Hom_{\mathbb{C}}(L,L\otimes V)$, by $B(\varphi\otimes l\otimes v)(l')=\varphi(l')l\otimes v, l'\in L$. The map $B$ is a $G$-isomorphism with respect to the following $G$-action on $\Hom_{\mathbb{C}}(L,L\otimes V)$: for $\Phi\in \Hom_{\mathbb{C}}(L,L\otimes V), \, l\in L, \, g\in G,\, (g\Phi)(l)=(g\otimes g).(\Phi(g^{-1}l)) $. Notice that the invariants with respect to this action are exactly the $G$-intertwining operators $L\to L\otimes V$. It is also interesting to note that the composite map $(A\otimes id)\circ B^{-1}: \Hom_{\mathbb{C}}(L,L\otimes V)\to O(G)\otimes V$ is the trace map; more precisely, for a basis $l_i$ of $L$, dual basis $\varphi_i$ of $L^*$, and $\Phi \in \Hom_{\mathbb{C}}(L,L\otimes V)$, its image $(A\otimes id) (B(\Phi))$ is a polynomial on $G$ given by  $$x\mapsto \mathrm{Tr}|_{L}(\Phi\circ x)= \sum_i (\varphi_i\otimes id) \Phi(xl_i)\in V.$$

To summarize, we have the following diagram: 
$$ \begin{CD}
O(G)\otimes V @<A\otimes id<< \bigoplus_{\mu \in P_{+}} L_{\mu}^*\otimes L_{\mu}\otimes V  @>B>> \bigoplus_{\mu \in P_{+}}  \Hom_{\mathbb{C}}(L_{\mu},L_{\mu} \otimes V)\\
@VV\Res V @VV\Res V @VV\Res V\\
O(H)\otimes V @<A\otimes id<< \bigoplus_{\nu \in P} \mathbb{C}_{\nu}^*\otimes \mathbb{C}_{\nu}\otimes V  @>B>> \bigoplus_{\nu \in P}  \Hom_{\mathbb{C}}(\mathbb{C}_\nu,\mathbb{C}_{\nu} \otimes V)\\
\end{CD} 
$$
where $\mathbb{C}_{\nu}$ denotes a one dimensional representation of $H$ on which $H$ acts by a character $\nu$. The isomorphisms $A$ and $B$ for $H$ are completely analogous to those for $G$. All $\Res$ maps are naturally defined. $\Res: O(G)\to O(H)$ is restricting a polynomial map to a subvariety. $\Res: \bigoplus_{\mu \in P_{+}} L_{\mu}^*\otimes L_{\mu}\to \bigoplus_{\nu \in P} \mathbb{C}_{\nu}^*\otimes \mathbb{C}_{\nu}$ corresponds to decomposing representations $L_{\mu}$ and $L_{\mu}^*$ into their $H$-isotypic components $L_{\mu,\nu}$ and $L_{\mu,\nu}^*$, then annihilating all parts  that are not diagonal, i.e. parts of the form $L_{\mu,\nu}^*\otimes L_{\mu,\eta}, \nu \ne \eta$, and finally taking a trace $L_{\mu,\nu}^*\otimes L_{\mu,\nu}\to \mathbb{C}_{\nu}^*\otimes \mathbb{C}_{\nu}$. In other words, for $ \varphi_{\nu}\otimes v_{\nu}\in \mathbb{C}_{\nu}^*\otimes \mathbb{C}_{\nu}$ a fixed basis with $\varphi_{\nu}(v_{\nu})=1$, and for $\varphi \in L_{\mu,\nu}^*$, $v\in L_{\mu,\nu}$, $v'\in L_{\mu,\nu'}$ with $\nu'\ne \nu$, the map is $\varphi \otimes v\mapsto \varphi(v)\varphi_{\nu}\otimes v_{\nu}$, and $\varphi \otimes v'\mapsto 0$.

This $\Res$ is defined to make the left square in the diagram commute. Analogously, the rightmost $\Res$ corresponds to viewing the homomorphisms as maps of $H$-representations, decomposing and  forgetting the non diagonal parts. The right square in the diagram also commutes. 

Theorem \ref{KNV2} can now be restated as follows:
\begin{corollary} For every $\mu \in P_{+}$, let $(\Phi_{\mu,j})_{j}$ be a basis of the space of intertwining operators $\Hom_{G}(L_{\mu},L_{\mu} \otimes V)$. For every such operator $\Phi_{\mu,j}$ define its trace function to be $\Psi_{\mu,j} \in O(H)\otimes V$, given by $\Psi_{\mu,j}(x)=\Tr| _{L_{\mu}}(\Phi_{\mu,j} \circ x)$. Then the set of all $\Psi_{\mu,j}$ is a basis of the space of functions in $O(H)\otimes V$ that satisfy 1)-3) from the statement of Theorem \ref{KNV2}.
\label{KNV3}
\end{corollary}

This is the form of the theorem that we will prove in the quantum case. Now let us illustrate this form of the theorem with a simple example where we can write everything explicitly. 

\begin{example}\label{primjer}
Let $\mathfrak{g}=\mathfrak{sl}_2$, $G=SL_{2}$. The rank of $\mathfrak{g}$ is $1$, so identify $\mathfrak{h}^*$ with $\mathbb{C}$ by $z \mapsto z\frac{\alpha}{2}$ for $\alpha$ the positive root. Then the dominant weights are identified with nonnegative integers. Let $V=L_{2}$ be the three dimensional irreducible representation with highest weight $2$. Pick a basis $v_{-2},v_{0},v_{2}$ of weight vectors for it, so that $v_{i}\in V[i]$, and $F.v_{2}=v_{0},F^2.v_{2}=v_{-2}$. Pick an analogous basis for any $L_{\mu}$, by $F^i.l_\mu, i=0,\ldots \mu$, and a dual basis to it $\varphi_i$, $\varphi_i(F^j.l_\mu)=\delta_{ij}$.

Let $\mu\in\mathbb{N}_{0}$ be arbitrary. Let us first describe all intertwining operators $$\Phi\in \Hom_{SL_2}(L_{\mu},L_{\mu}\otimes V)=\Hom_{\mathfrak{sl}_2}(L_{\mu},L_{\mu}\otimes V).$$ The map $\Phi$ is determined by $\Phi (l_{\mu})$, which needs to be a singular vector in $L_{\mu}\otimes V$ of total weight $\mu$. So,
$$\Phi (l_{\mu})=c_{0}l_{\mu}\otimes v_{0}+c_{1}F.l_{\mu}\otimes v_{2}.$$
The condition that this needs to be a singular vector in $L_{\mu}\otimes V$ gives a recursion on the coefficients $c_{i}$. In general (for any $\mathfrak{g}$ and any $V$), if $c_{0}=0$ then $\Phi=0$ (see \cite{EV1}, or Lemma \ref{iso} below). 
Scaling so that $c_{0}=1$ in this example we get
$$\Phi (l_{\mu})=l_{\mu}\otimes v_{0}-\frac{2}{\mu}F.l_{\mu}\otimes v_{2}.$$ The dimension of the space of $\mathfrak{g}$-intertwiners $L_{\mu}\to L_{\mu}\otimes V$ is $1$, except when $\mu=0$, when it is $0$. This also illustrates the general case, when for generic $\mu$ the spaces $\Hom_{\mathfrak{g}}(L_{\mu},L_{\mu}\otimes V)$ and $V[0]$ are isomorphic. The isomorphism sends $v\in V[0]$ to the $\mathfrak{g}$-intertwining operator $\Phi$ determined by  $\Phi (l_{\mu})=l_{\mu}\otimes v+\textrm{terms with first factor of lower weight}$ (again, see \cite{EV1} or Lemma \ref{iso} below). 

Set $h=h_1$. Any element of $\mathfrak{h}$ is of the form $zh, z\in \mathbb{C}$, and $\alpha(h)=2$. The trace function $\Psi$ on $H=\exp{\mathfrak{h}}$ is then, for $z\in \mathbb{C}$
\begin{eqnarray*}
\Psi(e^{zh}) & = & \Tr _{L_{\mu}}(\Phi\circ e^{zh})=\sum_{i=0}^{\mu} (\varphi_{i}\otimes id) (\Phi (e^{zh} (F^i. l_{\mu})) \\
& = & \sum_{i=0}^{\mu} e^{(\mu-i\alpha)(zh)} (\varphi_{i}\otimes id) (F^i. (l_{\mu}\otimes v_{0}-\frac{2}{\mu}F.l_{\mu}\otimes v_{2}))  \\
& = & \sum_{i=0}^{\mu} e^{(\mu-2i)z} (1-\frac{2}{\mu}i) v_{0}\\
& = &  \sum_{\mu-2i\ge0} \frac{\mu-2i}{\mu} (e^{z\cdot(\mu-2i)}- e^{-z\cdot(\mu-2i)}) v_{0}\\
\end{eqnarray*}

In this notation, $O(H)=\mathbb{C}[e^{z},e^{-z}]=span\{ e^{zh}\mapsto e^{nz}, n\in \mathbb{Z}\}$. As we vary $\mu\in \mathbb{N}$ and allow linear combinations of such trace functions, we can obviously get all the functions in $O(H)\otimes V$ of the form $f(e^{zh})=\sum_{n}a_{n}e^{nz}v_{0}$ that satisfy $a_{n}=-a_{-n} \, \forall n\in \mathbb{N}_{0}$. On the other hand, a function $f(e^{zh})=\sum_{n,i}a_{n,i}e^{nz}v_{i}$ in $O(H)\otimes V$ that satisfies 1)-3) must have:
\begin{enumerate}
\item $a_{n,i}=0$ unless $i=0$, so $f(e^{zh})=\sum_{n}a_{n}e^{nz}v_{0}$;
\item the Weyl group invariance: the Weyl group of $\mathfrak{sl}_2$ is $\mathbb{Z}_{2}$, and the nontrivial element acts on the $0$ weight subspace of $L_{2m}$ by $(-1)^m$; so in our case $f(e^{-zh})=-f(e^{zh})$, which means $a_{n}=-a_{-n}$;
\item $E.f(e^zh)=\sum_{n}a_{n}e^{nz}E.v_{0}=\sum_{n>0} a_n(e^{nz}-e^{-nz})\cdot 2 v_{2}$; every term $e^{nz}-e^{-nz}$ is divisible by $(1-e^{\alpha})(e^{zh})=1-e^{2z}$ in the ring $O(H)$ (and this condition is trivial in this case; see remark \ref{last}). 
\end{enumerate}
Here we can directly see these are the same spaces of functions, as claimed by the theorem.
\end{example}

\label{q=1}

\section{The generalized Chevalley isomorphism in the quantum case}
\label{3}

We keep the notation from Section \ref{q=1}:  $C=(a_{ij})$ is a Cartan matrix of finite type, $(\mathfrak{h}, \mathfrak{h}^*, \Pi=\{ \alpha_1,\ldots \alpha_{r} \}, \Pi^{\vee}=\{ h_{1},\ldots h_{r}\})$ is its realization, $d_{i}$ the symmetrizing integers, $\left<\cdot,\cdot \right>$  the form identifying $\mathfrak{h}$ and $\mathfrak{h}^*$, $H$ the complex torus with the map $\exp:\mathfrak{h}\to H$ whose kernel is $2i\pi$ times the dual weight lattice, $W$ the Weyl group, $P$ the weight lattice and $P_{+}$ the set of dominant integral weights. 

Let $q\in \mathbb{C}^{\times}$ not a root of unity. Pick $t\in \mathbb{C}$ such that $e^{t}=q$. For $x\in \mathbb{C}$, define $q^{x}=e^{tx}$. For $h\in \mathfrak{h}$, define $q^{h}=e^{th}\in H$, and for $\lambda \in \mathfrak{h}^*$ use the identification $\mathfrak{h}^*\cong \mathfrak{h}$ to define $q^{\lambda}=e^{t\lambda}\in H$. For a function $e^{\nu} \in O(H), \nu \in P$, we now have
 $$e^{\nu}(q^{h})=e^{\nu}(e^{th})=e^{t\nu(h)}=q^{\nu(h)},$$
 $$e^{\nu}(q^{\lambda})=e^{\nu}(e^{t\lambda})=e^{t\left<\nu,\lambda \right>}=q^{\left<\nu,\lambda \right>}.$$

To this data one may associate a quantum group $U_{q}(\mathfrak{g})$, and its representations. First define quantum integers as $[m]_{q}=\frac{q^m-q^{-m}}{q-q^{-1}}$, quantum factorials as $[m]_{q}!=[m]_{q}\cdot [m-1]_{q}\cdot \ldots [1]_{q}$, and $q_{i}=q^{d_{i}}=q^{\left<\alpha_{i},\alpha_{i}\right>/2}$. As an associative algebra, $U_{q}(\mathfrak{g})$ is given by generators $E_1,\ldots E_{r},F_1,\ldots F_{r},$ and $q^{h}, h\in\mathfrak{h}$ (here, $q^{h}$ is a formal symbol for a generator, meant to suggest how this element will act on weight spaces), with  relations
\begin{eqnarray*}
q^{h}q^{h'}=q^{h+h'} & &  [E_{i},F_{j}]=\delta_{ij}\frac{q_i^{h_{i}}-q_i^{-h_{i}}}{q_i-q_i^{-1}}\\
q^hE_{i}q^{-h}=q^{\alpha_i(h)}E_{i} & & q^hF_{i}q^{-h}=q^{-\alpha_i(h)}E_{i},\\
\end{eqnarray*} with $q_{i}^{h_{i}}=q^{d_{i}h_{i}}$, and Serre relations
$$\sum_{k=0}^{1-a_{ij}}\frac{(-1)^k}{[k]_{q_{i}}! [1-a_{ij}-k]_{q_{i}}! }E_{i}^{1-a_{ij}-k}E_{j}E_{i}^{k}=0$$
$$\sum_{k=0}^{1-a_{ij}}\frac{(-1)^k}{[k]_{q_{i}}! [1-a_{ij}-k]_{q_{i}}! } F_{i}^{1-a_{ij}-k}F_{j}F_{i}^{k}=0. $$

$U_{q}(\mathfrak{g})$ is a Hopf algebra, with the coproduct $\Delta$, counit $\varepsilon$, and the antipode $S$ given on the generators by
$$\begin{array}{ccc}
\Delta(q^{h})=q^{h}\otimes q^{h} & \Delta(E_{i})=E_{i}\otimes q_{i}^{h_{i}}+1\otimes E_{i} & \Delta(F_{i})=F_{i}\otimes 1+q_{i}^{-h_{i}}\otimes F_{i}\\
\varepsilon(q^{h})=1 & \varepsilon(E_{i})=0 & \varepsilon(F_{i})=0 \\
S(q^{h})=q^{-h} & S(E_{i})=-E_{i}q_{i}^{-h_{i}} &  S(F_{i})=-q_{i}^{h_{i}}F_{i}
\end{array}$$

Representations of $U_{q}(\mathfrak{g})$ that we are going to consider are going to be in category $\mathcal{O}$ and of type I. This means that a representation is a vector space $V$ with an algebra homomorphism $U_{q}(\mathfrak{g})\to End(V)$, such that the weight spaces $V[\nu]=\{ v\in V | q^h.v=q^{\nu(h)}v \, \, \forall h\in \mathfrak{h} \}, \nu\in P$, are all finite dimensional, $V=\bigoplus_{\nu \in P} V[\nu]$, and all weights appearing with nonzero weight space will be contained in a union of finitely many cones of the form $\{\nu-\sum_in_{i}\alpha_{i},\, n_{i}\in \mathbb{N}_0\}$ in $P$. Moreover, we will only be interested in finite dimensional representations and Verma modules, defined below. 

As $U_{q}(\mathfrak{g})$ is a Hopf algebra, its representations form  tensor category, as an element $X\in U_{q}(\mathfrak{g})$ acts on a tensor product of representations by $\Delta(X)$. We can also define duals of representations. For a finite dimensional $U_{q}(\mathfrak{g})$ module $V$, define its left dual ${^*V}$ to be the space of functionals on $V$ together with the $U_{q}(\mathfrak{g})$ action $(X\varphi)(v)=\varphi(S^{-1}(X)v)$. Left dual space ${^*V}$ comes with natural isomorphisms ${^*V}\otimes U\cong \Hom_{\mathbb{C}}(\bold{1}, {^*V}\otimes U)\cong \Hom_{\mathbb{C}}(V,U)$ for every module $U$. In the classical case of $U(\mathfrak{g})$, we have $S=S^{-1}$, as $S(X)=-X$ for $X\in \mathfrak{g}$, so left dual modules for quantum groups are one of two possible generalizations of the notion of dual module for enveloping algebras. The other one is the right dual  module, defined using $S$ instead of $S^{-1}$. 

For any $\mu \in P$, let $M_{\mu}$ denote the Verma module with highest weight $\mu$. It is a module generated over $U_{q}(\mathfrak{g})$ by a distinguished singular vector $m_{\mu}$, with relations $E_{i}m_{\mu}=0$, $q^{h}m_{\mu}=q^{\mu(h)}m_{\mu}$. If $\mu\in P_{+}$, then $M_{\mu}$ has a finite dimensional irreducible quotient; call it $L_{\mu}$, and call the image of $m_{\mu}$ in it $l_{\mu}$. As in the classical case, the finite dimensional irreducible representations we are interested in are labeled by integral dominant weights. We will mainly be interested in them, and occasionally use an auxiliary Verma module. 

\begin{remark}\label{qh}
Note that the symbol $q^{h}$ denotes both the element $\exp(th)$ of the group $H$ and the generator of $U_{q}(\mathfrak{g})$. This makes sense  because on any representation $V$ in category $\mathcal{O}$, these elements diagonalize with the same weight spaces, and act on such a weight space $V[\nu]$ with the same eigenvalues: $q^{h}\in U_{q}(\mathfrak{g)}$ acts by $q^{\nu(h)}$, $e^{th}\in H$ acts by $e^{\nu(th)}$, and $q^{\nu(h)}=e^{t\nu(h)}=e^{\nu(th)}$.

In other words, there exists a group homomorphism from the multiplicative group of all elements of the form $q^{h} \in U_{q}(\mathfrak{g})$ to $H$ given by $q^{h}\mapsto \exp(th)$. It is surjective, its kernel is the set of all $q^{h}\in U_{q}(\mathfrak{g}), h\in 2\pi i \mathbb{Z}\Pi^{\vee}/t$, and any representation in category $\mathcal{O}$ factors through this homomorphism. 
\end{remark}

\begin{remark}
Another way to define the setup we need is to avoid defining the quantum group $U_{q}(\mathfrak{g})$ altogether, and to instead just define its category $\mathcal{O}$ of representations. Namely, we define objects in the category to be $P$-graded vector spaces $V$ with graded pieces $V[\nu], \nu \in P$, together with operators $E_i,F_i$, such that:
\begin{itemize}
\item all $V[\nu]$ are finite dimensional:
\item the set of all $\nu$ with $V[\nu]\ne 0$ is contained in a union of finitely many cones of the form $\{ \lambda-\sum_{i}n_{i}\alpha_{i}| n_{i}\in \mathbb{N}_{0} \}$;
\item $E_{i}:V[\nu]\to V[\nu+\alpha_i]$, $F_{i}:V[\nu]\to V[\nu-\alpha_i]$;
\item $E_{i}, F_{i}$ satisfy Serre relations;
\item $[E_i, F_{j}]|_{V[\nu]}=\delta_{ij}\frac{q_i^{\nu(h_{i})}-q_i^{-\nu(h_{i})}}{q_i-q_i^{-1}} id |_{V[\nu]}$.
\end{itemize}
Morphisms in the category are morphisms of graded vector spaces that commute with the operators $E_{i}, F_{i}$. Tensor structure of the category can be defined by similar formulas. 

It is obvious that these two definitions of category $\mathcal{O}$ are equivalent. The first one is the usual definition of an algebra and its category of representations. The advantages of the second one are that it avoids the ambiguity of defining $U_{q}(\mathfrak{g})$, allows for a very clear restriction functor from this category to the category of representations of $H$, avoids the representations of $U_{q}(\mathfrak{g})$ which are not of type I, and is a more direct generalization of the the category of representations of $U(\mathfrak{g})$ that we considered, because just replacing $q$ by $1$ and $[m]_{q}$ by $m$ in all formulas gives exactly the category of representations of $U(\mathfrak{g})$ we considered.
\label{justcat}
\end{remark}

A practical consequence of the last remark is that there is a functor from the category of $U_{q}(\mathfrak{g})$ representations considered above to the category of representations of the torus $H$, given by letting $x\in H$ act on the space $V[\mu]$ by $e^{\mu}(x)\mathrm{id}_{V[\mu]}$. In light of Remark \ref{qh}, it corresponds to restricting the representation of $U_{q}(\mathfrak{g})$ to the subalgebra generated by all the $q^{h}$.

Inspired by the classical case, define functions on a quantum group to be $$O_{q}(G)=\bigoplus_{\mu \in P_{+}} {^*L_{\mu}\otimes L_{\mu}}.$$
This is a $U_{q}(\mathfrak{g})$ module with the usual action of $U_{q}(\mathfrak{g})$ on ${^*L_{\mu}}\otimes L_{\mu}$.

\begin{remark}
See \cite{KS} for a discussion on various equivalent definitions of quantized algebras of functions on a Lie group. In particular, in Chapter 3, Proposition 2.1.2 it is shown that one can define $O_{q}(G)$ as $\bigoplus_{\mu}{L_{\mu}\otimes L_{\mu}^*}$. This definition is equivalent to ours, as there is an isomorphism $L_\mu^*\to L_{\mu^*}$ whose dual is an isomorphism $L_\mu\to {^*L_{\mu^*}}$. Here $\mu^*$ is a dual weight to $\mu$, and can be calculated as $\mu^*=-w_{0}\mu$ for $w_{0}$ the longest element of $W$. The map $\mu\mapsto \mu^*$ is an involution on the set of dominant integral weights.
\end{remark}

This will be the setting for the rest of the paper. Also, let $V$ will be a finite dimensional representation of type I in category $\mathcal{O}$, that is a direct sum of finitely many $L_{\mu}$. As we are interested in describing restrictions of functions with values in $V$, which corresponds to taking tensor products with $V$, all the statements and conditions will behave nicely with respect to decomposing $V$  into direct sums. This means that the restriction  theorems will hold for $V$ if and only if they will hold for every direct summand of $V$. As a consequence, we can at any point assume $V$ is irreducible, and all the conclusions we make will hold for any $V$ that is a direct sum of (possibly infinitely many) irreducible finite dimensional modules. 

As in the classical case, we have the following diagram:
$$ \begin{CD}
O_{q}(G)\otimes V = \bigoplus_{\mu \in P_{+}} {^*L_{\mu}}\otimes L_{\mu}\otimes V  \cong \bigoplus_{\mu \in P_{+}}  \Hom_{\mathbb{C}}(L_{\mu},L_{\mu} \otimes V)\\
@VV\Res V \\
O(H)\otimes V \cong \bigoplus_{\nu \in P} {\mathbb{C}^*_{\nu}}\otimes \mathbb{C}_{\nu}\otimes V \cong \bigoplus_{\nu \in P}  \Hom_{\mathbb{C}}(\mathbb{C}_\nu,\mathbb{C}_{\nu} \otimes V)\\
\end{CD}
$$

Here we are using that on $H$, $S=S^{-1}$ so $^*\mathbb{C}_{\nu}=\mathbb{C}^*_{\nu}$. Let us list all the maps and all the actions of $U_{q}(\mathfrak{g})$ on these spaces; checking that all maps are isomorphisms of $U_{q}(\mathfrak{g})$ modules is then a direct computation.
\begin{itemize}
\item The map $\Res$, as in the classical case, corresponds to decomposing the representation $L=L_{\mu}$ into weight spaces, making an $H$-representation out of each weight space by defining $x|_{L[\nu]}=e^{\nu}(x)\mathrm{id}_{L[\nu]}$, annihilating the non-diagonal part and taking the trace. As in Remark \ref{justcat}, this corresponds to understanding a representation $L$ of a quantum algebra $U_{q}(\mathfrak{g})$ as an $H$-representation given by weight decomposition together with the operators $E_{i}:L[\nu]\to L[\nu+\alpha_i]$ and $F_{i}:L[\nu]\to L[\nu-\alpha_i]$. Alternatively, it corresponds to understanding $H$ as a multiplicative subgroup of $U_{q}(\mathfrak{g})$ like in Remark \ref{qh}.
\item The maps $ {^*L_{\mu}}\otimes L_{\mu}\otimes V  \to  \Hom_{\mathbb{C}}(L_{\mu},L_{\mu} \otimes V)$ and ${\mathbb{C}^*_{\nu}}\otimes \mathbb{C}_{\nu}\otimes V \to  \Hom_{\mathbb{C}}(\mathbb{C}_\nu,\mathbb{C}_{\nu} \otimes V)$ are natural, as they were in the classical case: $$\varphi\otimes l\otimes v\mapsto \big( l'\mapsto \varphi(l') l\otimes v \big)$$ for $l,l'\in L_{\mu}, \varphi\in {^*L_{\mu}}, v\in V$, or for $l,l'\in \mathbb{C}_{\nu}, \varphi\in {\mathbb{C}^*_{\nu}}, v\in V$. 
\item The map ${\mathbb{C}^*_{\nu}}\otimes \mathbb{C}_{\nu}\otimes V\to O(H)\otimes V$ is, as in the classical case, given by $$\varphi\otimes l\otimes v\mapsto \big( x\mapsto \varphi(xl)v\big)$$ for $\varphi \in {\mathbb{C}^*_{\nu}}, l \in \mathbb{C}_{\nu}, v\in V, x\in H$.
\item $U_{q}(\mathfrak{g})$ action on ${^*L_{\mu}}\otimes L_{\mu}\otimes V$ is the usual one on a triple tensor product, with $X\in U_{q}(\mathfrak{g})$ acting by $\Delta^{2}(X)=(\Delta\otimes id )\circ \Delta (X)=X_{(1)}\otimes X_{(2)}\otimes X_{(3)}$ in Sweedler's notation.
\item The $U_{q}(\mathfrak{g})$-action on $\Hom_{\mathbb{C}}(L_{\mu},L_{\mu} \otimes V)$ is as follows: for $\Phi \in \Hom_{\mathbb{C}}(L_{\mu},L_{\mu} \otimes V), X\in U_{q}(\mathfrak{g}), l\in L_{\mu}$, $$(X\Phi)(l)=\Delta(X_{(2)}).\Phi(S^{-1}(X_{(1)})l)=(X_{(2)}\otimes X_{(3)}).\Phi(S^{-1}(X_{(1)})l).$$
\end{itemize}

We are interested in the $U_{q}(\mathfrak{g})$ invariants in $O_{q}(G)\otimes V$ and their restrictions to $O(H)\otimes V$. The above action of $U_{q}(\mathfrak{g})$ on $\Hom_{\mathbb{C}}(L_{\mu},L_{\mu} \otimes V)$ is the usual one, so the space of invariants is exactly $$\Hom_{\mathbb{C}}(L_{\mu},L_{\mu} \otimes V)^{U_{q}(\mathfrak{g})}=\Hom_{U_{q}(\mathfrak{g})}(L_{\mu},L_{\mu} \otimes V).$$ 
We can again write explicitly the composite map $\Hom_{\mathbb{C}}(L_{\mu},L_{\mu} \otimes V)\to O(H)\otimes V$; it is the trace map. It maps $\Phi\in \Hom_{\mathbb{C}}(L_{\mu},L_{\mu} \otimes V)$ to the polynomial function $\Psi: H\to V$ given by:
$$\Psi (x)=Tr|_{L_{\mu}}(\Phi\circ x)\in V.$$ By further abuse of notation, we will call this map $\Res$, as well as its restriction to the space of invariants. 

We can now state the main theorem, analogous to Theorem \ref{KNV2}.
\MAIN
The proof of the theorem will be given in Section \ref{proof}. We will review the definition and some properties of the dynamical Weyl group in Section \ref{dyn}.

\section{Intertwining operators, dynamical Weyl group and trace functions}
\label{dyn}

Let us first fix some conventions for the rank one case $U_{q}( \mathfrak{sl}_2)$. In that situation, the Cartan matrix is $C=[2]$ and $\mathfrak{h}$ is one dimensional. As $r=1$, we will use the notation $h=h_{1}$ and $\alpha=\alpha_{1}$; we have $\alpha(h)=2$. Then we can identify $\mathfrak{h}^*$ with $\mathbb{C}$ by $n \leftrightarrow n\frac{\alpha}{2}$. Integral weights $P$ are thus identified with integers $\mathbb{Z}$ and dominant integral ones $P_+$ with nonnegative integers $\mathbb{N}_{0}$. 

Next, let us describe the notion of expectation value for general $C$ and $U_{q}(\mathfrak{g})$. Let $V$ be a finite dimensional representation and $\nu$ a weight of $V$. Any $\Phi \in \Hom_{U_{q}(\mathfrak{g})}(M_{\lambda},M_{\lambda-\nu}\otimes V)$ is of the form $\Phi(m_{\lambda})=m_{\lambda-\nu}\otimes v+\textrm{l.o.t.}$, where l.o.t. denotes the lower order terms, meaning terms with first coordinate in a lower weight space. Obviously, $v\in V[\nu]$. Define the expectation value of $\Phi$ to be $\left< \Phi \right>=v$. That means that if $\varphi_{\lambda-\nu}$ denotes an element of the (algebraic) dual of $M_{\lambda-\nu}$ that is $1$ on $m_{\lambda-\nu}$ and $0$ on all other weight spaces of $M_{\lambda-\nu}$, then $\left< \Phi \right>=(\varphi_{\lambda-\nu}\otimes id)(\Phi(m_{\lambda}))\in V[\nu]$. An analogous map exists for the situation when Verma modules are replaced by irreducible modules, and we will also write it as $\left< \, \cdot \, \right>:\Hom_{U_{q}(\mathfrak{g})}(L_{\lambda},L_{\lambda-\nu}\otimes V)\to V[\nu]$.

A morphism of $U_{q}(\mathfrak{g})$ modules $M_{\lambda}\to M_{\mu}\otimes V$ or $L_{\lambda}\to L_{\mu}\otimes V$ is clearly determined by the image of the highest weight vector, but for generic $\lambda$ even more is true: it is determined by the first term of it, more precisely by the expectation value. The precise statement is in the following lemma:

\begin{lemma}
\begin{enumerate}
\item For generic $\lambda$ and for $\lambda$ integral dominant with sufficiently large coordinates $\lambda(h_i)$, the expectation value maps $\left< \, \cdot \, \right>$ define isomorphisms
$$\Hom_{U_{q}(\mathfrak{g})}(M_{\lambda},M_{\lambda-\nu}\otimes V)\cong V[\nu]\cong \Hom_{U_{q}(\mathfrak{g})}(L_{\lambda},L_{\lambda-\nu}\otimes V).$$
\item For $\nu=0$ and $\lambda$ dominant integral, the image of the injective map $\Hom_{U_{q}(\mathfrak{g})}(L_{\lambda},L_{\lambda}\otimes V)\to V[0]$ is $$\{ v\in V[0] \, | \, E_{i}^{\lambda(h_{i})+1}v=0, \, \, i=1,\ldots r  \}.$$
\item For $U_{q}(sl_{2})$, $V=L_{2m}$, $\nu=0$ and $\lambda$ dominant integral, the expectation value map $$\Hom_{U_{q}(\mathfrak{g})}(L_{\lambda}, L_{\lambda}\otimes V) \to V[0]$$ is an isomorphism if and only if $\lambda \notin \{ 0,1,2,\ldots m-1\}$. If $\lambda \in \{ 0,1,2,\ldots m-1\}$, then $\Hom_{U_{q}(\mathfrak{g})}(L_{\lambda}, L_{\lambda}\otimes V)=0$. 
\end{enumerate}\label{iso}
\end{lemma}

\begin{proof}
\begin{enumerate}
\item For $\lambda$ generic or integral dominant with sufficiently large coordinates we have the following diagram:
$$\xymatrix{ 
\Hom_{U_{q}(\mathfrak{g})}(M_{\lambda},M_{\lambda-\nu}\otimes V) \ar[r] \ar[rd]   & \Hom_{U_{q}(\mathfrak{g})}(M_{\lambda},L_{\lambda-\nu}\otimes V) \ar[r]^{\simeq}  & \Hom_{U_{q}(\mathfrak{g})}(L_{\lambda},L_{\lambda-\nu}\otimes V) \ar@{_{(}->}[ld] \\
& V[\nu] \\
}$$

The map $\Hom_{U_{q}(\mathfrak{g})}(M_{\lambda},M_{\lambda-\nu}\otimes V) \to \Hom_{U_{q}(\mathfrak{g})}(M_{\lambda},L_{\lambda-\nu}\otimes V)$ is the composition with the projection map $M_{\lambda-\nu}\otimes V\to L_{\lambda-\nu}\otimes V$, and it is defined for any $\lambda$. In general, it is not injective. 

The map $\Hom_{U_{q}(\mathfrak{g})}(M_{\lambda},L_{\lambda-\nu}\otimes V) \to \Hom_{U_{q}(\mathfrak{g})}(L_{\lambda},L_{\lambda-\nu}\otimes V)$ is defined when all homomorphisms $M_{\lambda}\to L_{\lambda-\nu}\otimes V$ factor through $L_{\lambda}$. In particular, this happens if $\lambda$ is generic (in which case $M_{\lambda}=L_{\lambda}$ and the map is the identity), or when $\lambda-\nu$ is dominant integral (in which case $L_{\lambda-\nu}\otimes V$ is finite dimensional, so every map $M_{\lambda}\to L_{\lambda-\nu}\otimes V$ factors through the finite dimensional $L_{\lambda}$). In both of these cases, the map is an isomorphism. 

Both maps to $V[\nu]$ are the expectation value maps. 

Let us show that $\Hom_{U_{q}(\mathfrak{g})}(L_{\lambda},L_{\lambda-\nu}\otimes V)\to V[\nu]$ is injective. Pick a basis $v_{i}$ of  weight vectors for $V$. Let $\Phi  \ne 0 \in \Hom_{U_{q}(\mathfrak{g})}(L_{\lambda},L_{\lambda-\nu}\otimes V)$. Consider $\Phi (l_{\lambda})=\sum_{i} l_{i}\otimes v_{i}$ for some $l_{i}\in L_{\lambda-\nu}$. Because $\Phi (l_{\lambda})$ and all $v_{i}$ are weight vectors and $v_{i}$ are a basis, all the $l_{i}$ are weight vectors as well. Pick $l_{i_{0}}\ne 0$ with a highest weight among all nonzero $l_{i}$. Because $\Phi (l_{\lambda})$ is singular and $l_{i_{0}}$ has highest weight, $l_{i_{0}}$ is a singular vector in $L_{\lambda-\nu}$. Thus, $l_{i_{0}}=c\cdot l_{\lambda-\nu}$ for some $c\ne 0 \in \mathbb{C}$, and $\left< \Phi \right> =c\cdot v_{i_{0}}\ne 0$, so the expectation value map is injective.

Lemma 1 in \cite{EV1} states that for $\lambda$ generic, and in particular integral dominant with sufficiently large coordinates, the expectation value map $\Hom_{U_{q}(\mathfrak{g})}(M_{\lambda},M_{\lambda-\nu}\otimes V)\to V[\nu]$ is an isomorphism. The proof is straightforward, by noticing that the conditions on this map being an isomorphism are that a certain set of linear equations has a unique solution. It is a general argument of the type we used in Example \ref{primjer} for $\mathfrak{sl}_{2}$.

As the diagram from the beginning of the proof commutes, whenever the map $$\Hom_{U_{q}(\mathfrak{g})}(M_{\lambda},M_{\lambda-\nu}\otimes V)\to V[\nu]$$ is an isomorphism, the map $\Hom_{U_{q}(\mathfrak{g})}(L_{\lambda},L_{\lambda-\nu}\otimes V)\to V[\nu]$ is surjective and therefore also an isomorphism.

\item The map is injective due to proof of part (1). This proof also shows that with the assumptions of (2), namely $\nu=0$ and $\lambda$ dominant integral, $$\Hom_{U_{q}(\mathfrak{g})}(L_{\lambda},L_{\lambda}\otimes V)\cong \Hom_{U_{q}(\mathfrak{g})}(M_{\lambda},L_{\lambda}\otimes V).$$ The Verma module $M_{\lambda}$ is induced to $U_{q}(\mathfrak{g})$ from the subalgebra $U_{q}(\mathfrak{b}_{+})$, generated by all $q^{h}$ and $E_{i}$; the $U_{q}(\mathfrak{b}_{+})$ module we are inducing from is the one dimensional module $\mathbb{C}_{\lambda}$, with $q^{h}$ acting on it by $q^{\lambda (h)}\mathrm{id}$ and $E_{i}$ acting on it by $0$. So, $$\Hom_{U_{q}(\mathfrak{g})}(M_{\lambda},L_{\lambda}\otimes V)\cong \Hom_{U_{q}(\mathfrak{b}_{+})}(\mathbb{C}_{\lambda},L_{\lambda}\otimes V) \cong \Hom_{U_{q}(\mathfrak{b}_{+})}(\mathbb{C}_{\lambda}\otimes L_{\lambda}^*,V).$$

$L_{\lambda}^*$ is a lowest weight module with the lowest weight $-\lambda$. We can define the lowest weight analogue of Verma module $M^{-}_{-\lambda}$, which is induced from the module $\mathbb{C}_{-\lambda}$ over the subalgebra generated by all $q^h$ to the algebra $U_{q}(\mathfrak{b}_{+})$; so as a vector space it is isomorphic to the subalgebra $U_q(\mathfrak{n}_{+})$ generated by all the $E_{i}$. Call its lowest weight vector $\phi_{-\lambda}$. The module $L_{\lambda}^*$ is then known to be the quotient of $M^{-}_{-\lambda}$ by relations $E_{i}^{\lambda(h_{i})+1}\phi_{-\lambda}=0$.

Because of that, any $U_{q}(\mathfrak{b}_{+})$ map $\mathbb{C}_{\lambda}\otimes L_{\lambda}^*\to V$ is determined by the image of the lowest weight vector $1\otimes \phi_{-\lambda}$ in $V$. This must be a vector $v\in V$ of weight $\lambda-\lambda=0$, such that $E_{i}^{\lambda(h_{i})+1}.v=0$. It is clear that any such vector will define a $U_{q}(\mathfrak{b}_{+})$ intertwining operator $\mathbb{C}_{\lambda}\otimes L_{\lambda}^*\to V$.

The only thing left to notice is that under the isomorphism $$\Hom_{U_{q}(\mathfrak{b}_{+})}(\mathbb{C}_{\lambda}\otimes L_{\lambda}^*,V)\cong \Hom_{U_{q}(\mathfrak{g})}(L_{\lambda},L_{\lambda}\otimes V),$$ the vector $v$ from above corresponds to the expectation value of an intertwining operator $L_{\lambda}\to L_{\lambda}\otimes V$.

\item 
This follows directly from 2). $V[0]$ is one dimensional, so either the injective map $$\Hom_{U_{q}(\mathfrak{g})}(L_{\lambda},L_{\lambda}\otimes V)\to V[0]$$ is an isomorphism or the space $\Hom_{U_{q}(\mathfrak{g})}(L_{\lambda},L_{\lambda}\otimes V)$ is zero. As $d=1$, and after the identification $\mathfrak{h}^*\cong \mathbb{C}$ we have $\left< \lambda,\alpha \right>=\lambda \frac{\left< \alpha,\alpha \right>}{2}=\lambda$, part 2) of the lemma tells us that the image of the expectation value map is the set of $v\in V[0]$ such that $E^{\lambda+1}.v=0$. The maps $E:V[2i]\to V[2i+2] $ are injective for $i\ne m$, and $E^{\lambda+1}.v\in V[2\lambda +2]$, we conclude that the image of the map is zero unless $\lambda+1>m$, that is if $0\le \lambda \le m-1$. If $\lambda \ge m$, the set of such $v$ that $E^{\lambda+1}.v=0$ is the entire $V[0]$, so the injective map is an isomorphism.

This ends the proof, but it is interesting to note that the last case of $\lambda\in\{ 0,\ldots m-1\}$ is exactly when the commutative diagram from the beginning of this proof fails to be a commutative diagram of isomorphisms: $ \Hom_{U_{q}(\mathfrak{g})}(L_{\lambda},L_{\lambda}\otimes V)=0= \Hom_{U_{q}(\mathfrak{g})}(M_{\lambda},L_{\lambda}\otimes V)$; the spaces $V[0]$ and $ \Hom_{U_{q}(\mathfrak{g})}(M_{\lambda},M_{\lambda}\otimes V)\cong  \Hom_{U_{q}(\mathfrak{g})}(M_{\lambda},M_{-\lambda-2}\otimes V)$ are one dimensional, but the map between them is $0$.

\end{enumerate}
\end{proof}

\begin{remark}
Another way to prove (3) is to calculate explicitly the conditions on a vector in $L_{\lambda}\otimes V$ to be a singular vector of weight $\lambda$, and get a set of linear equations that have a solution if and only if $\lambda$ is in the above set. This is done in the first part of Theorem 7.1. in \cite{EV2}. 
\end{remark}

Following the notation in \cite{EV1}, for those $\lambda$ for which the expectation value map $$\left< \cdot \right> :  \Hom_{U_{q}(\mathfrak{g})}(M_{\lambda},M_{\lambda-\nu}\otimes V)\to V[\nu]$$ is an isomorphism, let $v\mapsto \Phi^{v}_{\lambda}$ be the inverse map; i.e. $ \Phi^{v}_{\lambda}$ is an intertwining operator such that $\left< \Phi^{v}_{\lambda} \right>=v$. For the same situation, let $\overline{\Phi}^{v}_{\lambda}$ be the intertwiner $L_{\lambda}\to L_{\lambda-\nu}\otimes V$ with $\left< \overline{\Phi}^{v}_{\lambda} \right>=v$.

The Weyl group $W$ is generated by simple reflections $s_i$ associated to simple roots $\alpha_{i}$. Let $\rho \in P\subseteq \mathfrak{h}^*$ be a weight such that $\rho (h_i)=1 \, \forall i$. Let the dot $w\cdot \lambda=w(\lambda+\rho)-\rho$ denote the shifted action of the Weyl group on $\mathfrak{h^*}$. The dynamical Weyl group of $V$ is a collection of operator valued functions $A_{w,V}(\lambda)$ labeled by $w\in W$, rational in $q^{\lambda}, \lambda \in \mathfrak{h}$, with $A_{w,V}(\lambda):V[\nu]\to V[w\cdot \nu]$. To define these operators, we first need a bit more notation and results from \cite{EV2}. Let $w=s_{i_1}\ldots s_{i_{l}}$ be a reduced decomposition of $w\in W$. Let $\lambda\in P_{+}$, and let $\alpha^{l}=\alpha_{i_{l}}$, $\alpha^{j}=(s_{i_{l}}\ldots s_{i_{j+1}})\alpha_{i_{j}}, \,\, j=1,\ldots l-1$. Let $n_{j}=2\left<\lambda+\rho,\alpha^{j} \right>/\left<\alpha^{j},\alpha^{j} \right>$. These are positive integers. Let $d^{j}=d_{i_j}$ be the symmetrizing numbers defined before. The following is Lemma 2 from \cite{EV1}:

\begin{lemma}
For $\lambda\in P_{+}$, the set of pairs $(n_{1},d^{1}), \ldots (n_{l},d^{l})$ and the product $F^{n_{1}}_{i_{1}} \ldots F^{n_{l}}_{i_{l}}$ don't depend on the reduced decomposition of $w\in W$. Hence, the vector $$m^{\lambda}_{w\cdot \lambda}=\frac{F^{n_{1}}_{i_{1}} \ldots F^{n_{l}}_{i_{l}}}{ [n_1]_{q^{d_{1}}}\ldots [n_l]_{q^{d_{l}}}}m_{\lambda} \in M_{\lambda}$$ is well defined. It is a singular vector of weight $w\cdot \lambda$.
\end{lemma}

We will use Proposition 15 and Corollary 16 from \cite{EV1} to define the dynamical Weyl group action. 

\begin{definition}
Let $v\in V[\nu]$, $w\in W$, $\lambda\in P_{+}$ with large enough coordinates compared with $\nu$. We have $$\Phi^{v}_{\lambda}(m_{\lambda})=m_{\lambda-\nu}\otimes v+\textrm{l.o.t.}.$$
Define $A_{w,V}(\lambda)v\in V[w\cdot \nu]$ by 
$$\Phi^{v}_{\lambda}(m^{\lambda}_{w\cdot \lambda})=m^{\lambda-\nu}_{w\cdot (\lambda-\nu)}\otimes A_{w,V}(\lambda)v+\textrm{l.o.t.}$$ (The proof that this is well defined, i.e. that the vector $\Phi^{v}_{\lambda}(m^{\lambda}_{w\cdot \lambda})$ is of that form, is in \cite{EV1}).

The operators $A_{w,V}(\lambda)$, defined for $\lambda$ dominant integral with large enough coordinates, depend rationally on $q^{\lambda}$ (in the sense that their coefficients in any basis are rational functions of $q^{\lambda(h_{i})}$), so they can be uniquely extended to rational functions of $q^{\lambda}$, for $\lambda\in \mathfrak{h}^*$. 
\end{definition}

The operators $A_{w,V}(\lambda)$ do not, in general, define a representation of the Weyl group. However, we have a weaker result below (Lemma 17 and Corollary 29 from \cite{EV1}). Let $l$ be the length function on the Weyl group $W$, defined to be the length of the shortest reduced expression. 
\begin{proposition}
\begin{enumerate}
\item If $w_1, w_2\in W$ such that $l(w_1w_2)=l(w_1)+l(w_2)$, then $$A_{w_1w_2,V}(\lambda)=A_{w_1,V}(w_2\cdot \lambda)A_{w_2,V}(\lambda).$$
\item Restrictions of operators $A_{w,V}(\lambda)$ to $V[0]$ satisfy $$A_{w_1w_2,V}(\lambda)=A_{w_1,V}(w_2\cdot \lambda)A_{w_2,V}(\lambda)$$ without any requirements on the length of $w_{i}\in W$. 
\end{enumerate}\label{isitarep}
\end{proposition}

For $U_{q}(\mathfrak{sl}_2)$ and $V$ a simple finite dimensional module, $V[0]$ is either $0$ (if $V=L_{2m+1}$) or one dimensional (if $V=L_{2m}$). In the latter case, the operators $A_{V}(\lambda)$ restricted to $V[0]$ are just rational functions of $q^{\lambda}$ times the identity operator on $V[0]$. We can calculate them explicitly:
\begin{lemma}\label{explicitA}
For $U_{q}(\mathfrak{sl}_2)$, $V=L_{2m}$, and $s$ the nontrivial element of the Weyl group $W=\mathbb{Z}_{2}$,
 $$A_{s,V}(\lambda)=(-1)^m\prod_{j=1}^{m}\frac{[\lambda+1+j]_{q}}{[\lambda+1-j]_{q}} \mathrm{id}_{V[0]}.$$
 \end{lemma}
\begin{proof}
Follows directly from Corollary 8 (iii) and Proposition 12 in \cite{EV1}. 
\end{proof}

One can now define two actions of the dynamical Weyl group on rational functions of $q^{\lambda}$ with values in $V[0]$:
\begin{definition}
\begin{enumerate}
\item The shifted action is given by $$(w\circ f)(\lambda)=A_{w,V}(w^{-1}\cdot \lambda)f(w^{-1}\cdot \lambda).$$
\item Define $\mathcal{A}_{w,V}(\lambda)=A_{w,V}(-\lambda-\rho).$
\item The unshifted action is given by $$(w * f)(\lambda)=\mathcal{A}_{w,V}(w^{-1}\lambda)f(w^{-1}\lambda).$$
\end{enumerate}
\end{definition}

\begin{corollary}\label{calA}
Restricted to $V[0]$, the operators $\mathcal{A}_{w,V}(\lambda):V[0]\to V[0]$ satisfy $$\mathcal{A}_{w_1w_2,V}(\lambda)= \mathcal{A}_{w_1,V}(w_2 \lambda) \mathcal{A}_{w_2,V}(\lambda).$$ 
\end{corollary}

\begin{remark}
In general, the shifted and the unshifted action are defined for rational functions with values in $V$. Because of Proposition \ref{isitarep}, in that case they don't define a representation of the Weyl group $W$, but define an action of a braid group of $W$. However, we will need them only for functions with values in $V[0]$, where both actions define a representation of $W$ (again due to Proposition \ref{isitarep}).
\end{remark}

The statement of the main theorem, \ref{main}, refers to the unshifted dynamical action from this definition. Here one must remember that we can use the form $\left<\cdot,\cdot \right>$ to identify $\mathfrak{h}\cong\mathfrak{h}^*$, so this definition of functions on $\mathfrak{h}^*$ can be applied to functions on $\mathfrak{h}$. With that identification, the part of the theorem ``$f\in O(H)\otimes V$ invariant under the unshifted action of dynamical Weyl group" means that for every $w\in W , \,\lambda\in \mathfrak{h}^*$, $$f(q^{2w\lambda})=\mathcal{A}_{w,V}(\lambda)f(q^{2\lambda}).$$

To prove the dynamical Weyl group invariance, we need to invoke several more definitions and results form \cite{EV1} and \cite{EV2}.

Remember that for $\mu$ large dominant and $v\in V[0]$ we defined $\Phi^{v}_{\mu}\in \Hom_{U_{q}(\mathfrak{g})}(M_{\mu},M_{\mu}\otimes V)$ such that $\left< \Phi^{v}_{\mu} \right>=v$, and analogously  $\overline{\Phi}^{v}_{\mu}\in \Hom_{U_{q}(\mathfrak{g})}(L_{\mu},L_{\mu}\otimes V)$ such that $\left< \overline{\Phi}^{v}_{\mu} \right>=v$. We also defined their trace functions. To introduce notation of \cite{EV1}, for $\lambda \in \mathfrak{h}^*$, define 
$$\Psi^{v}(\lambda,\mu)=\Tr|_{M_{\mu}}(\Phi^{v}_{\mu}q^{2\lambda}) \in V[0].$$ 
The functions we are interested in are
$$\Psi^{v}_{\mu}(\lambda)=\Tr|_{L_{\mu}}(\overline{\Phi}^{v}_{\mu}q^{2\lambda}) \in V[0].$$ 
The paper \cite{EV1} also uses generating functions for these trace functions. Pick a basis $v_{i}$ of $V[0]$ and let $v_{i}^*\in V^*[0]$ be the dual basis. Then define the generating functions as
$$\Psi_{V}(\lambda,\mu)=\sum_{i} \Psi^{v_{i}}(\lambda,\mu)\otimes v_{i}^* \in V[0]\otimes V^*[0] \cong \Hom_{\mathbb{C}}(V[0],V[0])$$
$$\Psi^{\mu}_{V}(\lambda)=\sum_{i} \Psi^{v_{i}}_{\mu}(\lambda)\otimes v_{i}^* \in V[0]\otimes V^*[0].$$

We are interested in functions of the type $f(q^{2\lambda})=\Psi^{v}_{\mu}(\lambda)$. More results are available about functions $\Psi^{v}(\lambda,\mu)$. Fortunately, there is a theorem allowing us to translate results of one type to another, analogous to Weyl character formula and proved as Proposition 42 in \cite{EV1}:
\begin{proposition}\label{42}
$\Psi^{v}_{\mu}(\lambda)=\sum_{w\in W}(-1)^{w} \Psi^{v}(\lambda,w\cdot \mu)A_{w,V}(\mu)$.
\end{proposition}

Let $\delta_{q}(\lambda)$ be the Weyl denominator $\delta_{q}(\lambda)=\sum_{w\in W}(-1)^{w}q^{2\left< \lambda,w\rho \right>}$. It satisfies 
\begin{lemma}\label{sign}
$\delta_{q}(w\lambda)=(-1)^w\delta_{q}(\lambda)$.
\end{lemma}
\begin{proof} It follows directly from the $W$-invariance of the form $\left< \cdot, \cdot \right>$.
\end{proof}

For finite dimensional $U_{q}(\mathfrak{g})$ modules $U,V$, define the fusion matrix $J_{UV}(\lambda):U\otimes V\to U\otimes V$ as follows. For generic $\lambda$ and $v\in V[\mu], u\in U[\nu]$, it is an operator such that $$(\Phi^{u}_{\lambda-\mu}\otimes 1)\circ \Phi^{v}_{\lambda}=\Phi^{J_{UV}(\lambda)(u\otimes v) }_{\lambda}.$$ It is a rational function of $q^{\lambda}$, and an invertible operator (see \cite{EV1}, Section 2.6).

If $J_{U,{^*U}}(\lambda)=\sum_{i} c_{i}\otimes c_{i}' $, with $c_{i}\in \End(U), c_{i}'\in \End({^*U})$, define $Q_{U}(\lambda)=\sum_{i}(c_{i}')^*c_{i}\in \End(U)$ (see \cite{EV2}). Use these to define the renormalized trace functions $$F_{V}(\lambda,\mu)=\delta_{q}(\lambda)\Psi_{V}(\lambda,-\mu-\rho)Q_{V}^{-1}(-\mu-\rho).$$ These satisfy (see Proposition 45 in \cite{EV1}):
\begin{proposition}\label{45}
$F_{V}(\lambda,\mu)=\big( \mathcal{A}_{w,V}(w^{-1}\lambda)\otimes \mathcal{A}_{w,V^*}(w^{-1}\mu)\big) F_{V}(w^{-1}\lambda,w^{-1}\mu)$.
\end{proposition}

These operators appear in many formulas because they transform the action of operators $A_{w,V}(\lambda)$ on the space $V$ and its duals. One of these, a special case of Proposition 20 in \cite{EV1}, is the following proposition:
\begin{proposition} \label{20}
When restricted to $V[0]$, $A_{w,V^*}(\lambda)^*=Q_{V}(\lambda)A_{w,V}(\lambda)^{-1}Q_{V}(w\cdot \lambda)^{-1}.$
\end{proposition}

\label{terribledef}

\section{Proof of Theorem \ref{main}}
\label{proof}

As we identified $(O_{q}(G)\otimes V)^{U_{q}(\mathfrak{g})}\cong \bigoplus_{\mu \in P_{+}}\Hom_{U_{q}(\mathfrak{g})}(L_{\mu},L_{\mu}\otimes V)$ and are thus interested in the map $\Res: \Hom_{U_{q}(\mathfrak{g})}(L_{\mu},L_{\mu}\otimes V)\to O(H)\otimes V$, all the claims can be stated and proved in this language of traces of intertwining operators. The main Theorem \ref{main} can be restated in this language as the following theorem, analogous to Corollary \ref{KNV3}.

\begin{theorem}\label{proving}
For any intertwining operator $\Phi\in \Hom_{U_{q}(\mathfrak{g})}(L_{\mu},L_{\mu} \otimes V)$ define its weighted trace  as a function $\Psi \in O(H)\otimes V$ given by $\Psi(x)=\Tr _{L_{\mu}}(\Phi \circ x)$. Then the map $$\Res: \bigoplus_{\mu\in P_{+}}\Hom_{U_{q}(\mathfrak{g})}(L_{\mu},L_{\mu} \otimes V)\to  O(H)\otimes V$$ given by $\Res \Phi=\Psi$ is injective, and its image consists of all the functions $f\in O(H)\otimes V$ that satisfy
 \begin{enumerate}
\item $f\in O(H)\otimes V[0]$;
\item $f$ is invariant under the (unshifted) action of the dynamical Weyl group, meaning that for all $w\in W, \, \lambda \in \mathfrak{h}^*$, $$f(q^{2w\lambda})=\mathcal{A}_{w,V}(\lambda)f(q^{2\lambda});$$
\item for every $\alpha_i\in \Pi$ and every $n\in \mathbb{N}$, the polynomial $E_{i}^n .f $ is divisible by $$(1-q_{i}^2e^{\alpha_i})(1-q_{i}^4e^{\alpha_i})\ldots (1-q_{i}^{2n}e^{\alpha_i}).$$ 
\end{enumerate}
\end{theorem}

\begin{lemma}\label{1)}
Trace functions $\Psi=\Res \Phi$ satisfy 1), i.e. $\Psi \in O(H)\otimes V[0]$.
\end{lemma}
\begin{proof}
Let $\Phi \in \Hom_{U_{q}(\mathfrak{g})}(L_{\mu},L_{\mu}\otimes V)$. We can assume we are calculating the trace of $\Phi$ using a basis of weight vectors in $L$. The image of every weight vector $l$ in $L_{\mu}$ under $\Phi$ is going to be a weight vector of the same weight, so when we write it as a sum of elementary tensors and pick the elementary tensor whose first component is $l$, the second component is going to have weight $0$. 
\end{proof}

\begin{lemma} \label{2)}
Trace functions $\Psi=\Res \Phi$ satisfy 2), i.e. for every $w\in W,\, \lambda \in \mathfrak{h}^*$, $$\Psi(q^{2w\lambda})=\mathcal{A}_{w,V}(\lambda)\Psi(q^{2\lambda}).$$
\end{lemma}
\begin{proof}
Using the definition of renormalized trace functions from section \ref{terribledef}, Proposition \ref{45}, definition of shifted and unshifted action of dynamical Weyl group, Proposition \ref{20}, defintion of renormalized trace functions again, and finally Lemma \ref{sign}, we get 
\begin{eqnarray*}
\Psi_{V}(\lambda,\mu)&=&\delta_{q}(\lambda)^{-1}F_{V}(\lambda,-\mu-\rho)Q_{V}(\mu)\\
&=&  \delta_{q}(\lambda)^{-1}\mathcal{A}_{w,V}(w^{-1}\lambda)F_{V}(w^{-1}\lambda,w^{-1}(-\mu-\rho))\mathcal{A}_{w,V^*}(w^{-1}(-\mu-\rho))^* Q_{V}(\mu)\\
&=& \delta_{q}(\lambda)^{-1}A_{w,V}(-w^{-1}\lambda-\rho)F_{V}(w^{-1}\lambda,w^{-1}(-\mu-\rho))A_{w,V^*}(w^{-1}\cdot \mu)^* Q_{V}(\mu) \\
&=&  \delta_{q}(\lambda)^{-1}A_{w,V}(-w^{-1}\lambda-\rho)F_{V}(w^{-1}\lambda,-w^{-1}(\mu+\rho))Q_{V}(w^{-1}\cdot \mu)A_{w,V}(w^{-1}\cdot \mu)^{-1} \\
&=&  \delta_{q}(\lambda)^{-1} \delta_{q}(w^{-1}\lambda)A_{w,V}(-w^{-1}\lambda-\rho)\Psi_{V}(w^{-1}\lambda, w^{-1}(\mu+\rho)-\rho)\cdot \\
& & \qquad \qquad \qquad \qquad \cdot Q_{V}(w^{-1}(\mu+\rho)-\rho)^{-1}Q_{V}(w^{-1}\cdot \mu)A_{w,V}(w^{-1}\cdot \mu)^{-1} \\
&=&(-1)^{w} A_{w,V}(-w^{-1}\lambda-\rho)\Psi_{V}(w^{-1}\lambda, w^{-1}\cdot \mu) A_{w,V}(w^{-1}\cdot \mu)^{-1}.
\end{eqnarray*}

As we are interested in traces of intertwining operators on irreducible modules and not on Verma modules, we use Poposition \ref{42} to translate the above identity to those functions:
\begin{eqnarray*}
\Psi^{v}_{\mu}(\lambda)&=&\sum_{w\in W}(-1)^w\Psi_{V}(\lambda,w\cdot \mu)A_{w,V}(\mu)\\
&=&\sum_{w\in W}(-1)^w(-1)^w A_{w,V}(-w^{-1}\lambda-\rho)\Psi_{V}(w^{-1}\lambda, w^{-1}\cdot (w\cdot \mu)) A_{w,V}(w^{-1}(w \cdot \mu))^{-1}  A_{w,V}(\mu)\\
&=&\sum_{w\in W}A_{w,V}(-w^{-1}\lambda-\rho)\Psi_{V}(w^{-1}\lambda, \mu)\\
&=&\sum_{w\in W}\mathcal{A}_{w,V}(w^{-1}\lambda)\Psi_{V}(w^{-1}\lambda, \mu).
\end{eqnarray*}

Finally, we use this and Corollary \ref{calA} to conclude that for any $w'\in W$,
\begin{eqnarray*}
\mathcal{A}_{w',V}(\lambda)\Psi^{v}_{\mu}(\lambda)&=&\sum_{w\in W}\mathcal{A}_{w',V}(\lambda)\mathcal{A}_{w,V}(w^{-1}\lambda)\Psi_{V}(w^{-1}\lambda, \mu)\\
&=&\sum_{w\in W}\mathcal{A}_{w'w,V}(w^{-1}\lambda)\Psi_{V}(w^{-1}\lambda, \mu)\\
&=&\sum_{w\in W}\mathcal{A}_{w,V}(w^{-1} w' \lambda)\Psi_{V}(w^{-1} w' \lambda, \mu)\\
&=& \Psi^{v}_{\mu}(w'\lambda)
\end{eqnarray*}
as required.
This proves the lemma with the above convention $\Psi(q^{2\lambda})=\Tr|_{L_{\mu}}(\Phi \circ q^{2\lambda})=\Psi^{v}_{\mu}(\lambda)$. Notice that we used the fact that $\Phi^{v}_{\mu}$ span the space $\bigoplus_{\mu}\Hom_{U_{q}(\mathfrak{g})}(L_{\mu},L_{\mu}\otimes V)$, so it is enough to prove the invariance for $\Psi^{v}_{\mu}$.
\end{proof}

\begin{lemma}\label{3)}
Trace functions $\Psi$ satisfy 3), i.e.  for every $i=1,\ldots r$ and every $n\in \mathbb{N}$, the polynomial $E_{i}^n .\Psi $ is divisible by $(1-q_{i}^2e^{\alpha_i})(1-q_{i}^4e^{\alpha_i})\ldots (1-q_{i}^{2n}e^{\alpha_i})$.
\end{lemma}
\begin{proof}
If $\Phi \in  \Hom_{U_{q}(\mathfrak{g})}(L_{\mu},L_{\mu}\otimes V)$, we can define its trace function not only as a function on $H$, but on the entire $U_{q}(\mathfrak{g})$, by $f(X)=\Tr|_{L_{\mu}}(\Phi\circ X)$. The restriction of $f$ to the subalgebra generated by all the $q^{h}$ is the trace function $\Psi$ as in the claim of the lemma. 

This defines a map from $\Hom_{\mathbb{C}}(L,L\otimes V)\cong {^*L}\otimes L\otimes V$ to linear functions from $U_{q}(\mathfrak{g})$ to $V$. We can make $U_{q}(\mathfrak{g})$ act on the algebraic dual of $U_{q}(\mathfrak{g})$ tensored with $V$ in a way to make the above defined trace map a morphism of $U_{q}(\mathfrak{g})$ modules. The easiest way to do that is to remember that the compatible definition of action of $Y\in U_{q}(\mathfrak{g})$ on ${^*L}\otimes L\otimes V$ was by $Y_{(1)}\otimes Y_{(2)}\otimes Y_{(3)}$, and to notice that the map from ${^*L}\otimes L\otimes V$ to the linear functions $f:U_{q}(\mathfrak{g})\to V$ corresponding to the one we just defined above is $\varphi\otimes l\otimes v\mapsto (X\mapsto \varphi(Xl)v)$. From this it is clear that the action of $Y\in U_{q}(\mathfrak{g})$ on the space of linear functions $f:U_{q}(\mathfrak{g})\to V$  is $$(Yf)(X)=Y_{(3)}.f(S^{-1}(Y_{(1)})XY_{(2)}).$$ 

If the function $f:U_{q}(\mathfrak{g})\to V$ was defined as a trace $f(X)=\Tr(\Phi\circ X)$ of an intertwining operator $\Phi$, then it is invariant with respect to the above action, and hence satisfies $$\varepsilon(Y) f(X)=Y_{(3)} .f(S^{-1}(Y_{(1)})XY_{(2)}).$$ Specializing this identity to $Y=q^{h}$, for which $\varepsilon(q^{h})=1$, $S^{-1}(q^h)=q^{-h}$, and $\Delta^{2} q^{h}=q^{h}\otimes q^{h}\otimes q^{h},$ we get that $f$ satisfies $$q^{h}.f(X)=f(q^{h}Xq^{-h}).$$
Specializing it to $Y=E_{i}$ instead, for which $\varepsilon(E_{i})=0$, $S^{-1}(E_{i})=-q_{i}^{-h_{i}}E_{i}$ and $$\Delta^{2} E_{i}=E_{i}\otimes q_{i}^{h_i}\otimes q_{i}^{h_i}+1\otimes E_{i}\otimes q_{i}^{h_{i}}+1\otimes 1\otimes E_{i},$$ we get that $f$ also satisfies $$0=q_{i}^{h_i}.f(-q_{i}^{-h_i}E_{i}Xq_{i}^{h_i})+q_{i}^{h_i}.f(XE_{i})+E_{i}.f(X)$$ so $$E_{i}.f(X)=f(E_{i}X)-f(q_{i}^{h_i}XE_{i}q_{i}^{-h_i}).$$

Using this formula, we will now prove by induction on $n$ that $$E_{i}^n.f(q^{h})=(1-q_{i}^{2}q^{\alpha_{i}(h)})\cdot (1-q_{i}^{4}q^{\alpha_{i}(h)})\ldots (1-q_{i}^{2n}q^{\alpha_{i}(h)})f(E_{i}^nq^{h}).$$ 

For $n=0$ the claim is trivial. Assume that it is true for $n-1$ and calculate
\begin{eqnarray*}
E_{i}^n.f(q^{h})&=&(1-q_{i}^{2}q^{\alpha_{i}(h)})\ldots (1-q_{i}^{2(n-1)}q^{\alpha_{i}(h)})E_{i}.f(E_{i}^{n-1}q^{h})\\
&=& (1-q_{i}^{2}q^{\alpha_{i}(h)})\ldots (1-q_{i}^{2(n-1)}q^{\alpha_{i}(h)}) \big( f(E_{i}^nq^{h})-f(q_{i}^{h_{i}}E_{i}^{n-1}q^{h}E_{i}q_{i}^{-h_{i}} ) \big)\\
&=& (1-q_{i}^{2}q^{\alpha_{i}(h)})\ldots (1-q_{i}^{2(n-1)}q^{\alpha_{i}(h)}) \big( f(E_{i}^nq^{h})-q_i^{2n}q^{\alpha_{i}(h)}f(E_{i}^nq^{h}) \big)\\
&=& (1-q_{i}^{2}q^{\alpha_{i}(h)})\ldots (1-q_{i}^{2n}q^{\alpha_{i}(h)}) f(E_{i}^nq^{h}).
\end{eqnarray*}

This ends the induction and proves the lemma. 
\end{proof}

\begin{remark}
\cite{KS}, Chapter 3, Definition 1.2.1, defines $O_{q}(G)$ as the Hopf subalgebra of the space of linear functionals on $U_{q}(\mathfrak{g})$ generated by the matrix elements of the finite dimensional representations of type I. The map from the beginning of the above proof, associating to $\varphi\otimes l\in {^*L_{\lambda}}\otimes L_{\lambda}$ the functional on $U_{q}(\mathfrak{g})$ given by $X\mapsto \varphi (Xl)$, is exactly the isomorphism from $\oplus_{\lambda}{^*L_{\lambda}}\otimes L_{\lambda}$ to this subalgebra of linear functionals, establishing the equivalence of these two definitions.
\end{remark}

\begin{lemma}\label{inj}
The restriction map $(O_{q}(G)\otimes V)^{U_{q}(\mathfrak{g})}\to O(H)\otimes V$ is injective.
\end{lemma}
\begin{proof}

We will prove that the map $$\bigoplus_{\mu\in P_+}\Hom_{U_{q}(\mathfrak{g})}(L_{\mu},L_{\mu}\otimes V)\to O(H)\otimes V$$ associating to the intertwining operator $\Phi$ its weighted trace $\Psi(x)=\Tr|_{L_{\mu}}(\Phi\circ x)$ is injective.

Let $$\Phi=\sum_{\mu}\Phi_{\mu}^{v_{\mu}}$$ be an element of $\bigoplus_{\mu\in P_+}\Hom_{U_{q}(\mathfrak{g})}(L_{\mu},L_{\mu}\otimes V)$ whose weighted trace is zero. Notice that $$\Tr(\Phi_{\mu}^{v_{\mu}} \circ x)=\sum_{\nu}u_{\mu,\nu}e^{\nu}(x)$$ for some $u_{\mu,\nu}\in V[0]$, with $u_{\mu,\nu}=0$ unless $\nu$ is a weight of $L_{\mu}$, and with $u_{\mu,\mu}=v_{\mu}$. The fact that $\Phi$ maps to zero can be written as $$\sum_{\mu}\sum_{\nu}u_{\mu,\nu}e^{\nu}=0.$$ 

Assume $\Phi\ne 0$. Then one can pick $\mu=\mu_{0}$ so that $\Phi_{\mu}^{v_{\mu}}$ is nonzero and $\mu_{0}$ is a highest weight with that property. Using the fact that $e^{\nu}$ are linearly independent, the coefficient with $e^{\mu_{0}}$ in the above equation is $$0=u_{\mu_{0},\mu_{0}}=v_{\mu_{0}}.$$ But then $0=\Psi_{\mu_{0}}^{v_{\mu_{0}}}$, contrary to the choice of $\mu_{0}$. So, $\Phi=0$ and the map is injective.
\end{proof}

\begin{lemma}\label{sl2}
Theorem \ref{proving} holds for $U_{q}(\mathfrak{sl}_2)$.
\end{lemma}
\begin{proof}
As stated before, we are using identifications $\mathbb{C}\cong \mathfrak{h}^*$, $z\mapsto z\frac{\alpha}{2}$. Let us use a slightly different convention for $\mathfrak{h}$: it is also one dimensional, so we can write any element of it as $zh$, for $h=h_{1}$ the standard generator of $\mathfrak{h}$ and $z\in \mathbb{C}$. The space $V[0]$ is one dimensional, so pick any $v_{0}\ne 0 \in V[0]$ and identify $V[0]\cong \mathbb{C}$ by it. The polynomial functions in $O(H)\otimes V[0]$ we talk about are, with all these identifications, $\mathbb{C}[q^{z},q^{-z}]$, spanned by functions $e^{n\alpha/2}(q^{zh})=q^{nz},\, n\in \mathbb{Z}$.  With all these conventions, $zh\leftrightarrow z\alpha=2z\frac{\alpha}{2}$, so and the above definitions of trace functions give $\Psi(q^{zh})=\Psi^{v}_{\mu}(z\frac{\alpha}{2})$. This is a good convention because $z\alpha/2 \leftrightarrow z$. The dynamical Weyl group invariance, with all these identifications, has the form $$\mathcal{A}_{s,V}(z)\Psi(q^{zh})=\Psi(q^{-zh}).$$

We are proving that two subspaces of polynomial functions $O(H)\otimes V$ are equal: the space of traces of intertwining operators and the space of functions satisfying 1)-3) from the statement of Theorem \ref{proving}. Lemmas \ref{1)}, \ref{2)} and \ref{3)} show that the space of traces of intertwining operators is contained in the space of functions satisfying 1)-3). We will now prove this lemma by proving that these two spaces of functions are of the same size (more accurately, as they are infinite dimensional, that there is a filtration on $O(H)$ such that dimensions match on every filterered piece; the filtration we use will be the obvious filtration by degree of a polynomial).

As stated in Lemma \ref{iso}, the space of intertwining operators $L_{\mu}\to L_{\mu}\otimes V$ for $V=L_{2m}$ is zero if $\mu=0,1,\ldots m-1$, and is one dimensional if $\mu \in \mathbb{N}, \mu \ge m$. The trace of such an operator is a Laurent polynomial $\Psi(x)=\Tr|_{L_{\mu}}(\Phi\circ x)=\sum_{\nu}a_{\nu}e^{\nu} (x)$, with all the $\nu$ that appear being weights of $L_{\mu}$. So, it is a Laurent polynomial of maximal (positive and negative) degree $\mu$. Using Lemma \ref{inj} that allows us to calculate the dimension of the space of trace functions by calculating the dimension of the appropriate space of intertwining operators, we can conclude that for any large enough positive integer $N$, the space of trace functions of maximal (positive) degrees less or equal to $N$ has dimension $N-m+1$. 

Now, let us calculate the dimension of the space of functions that satisfy 1)-3) and have degree $\le N$. It is enough to show that it has dimension less or equal to $N-m+1$; from this it will follow that it has exactly this dimension and that the two spaces are equal. 

Let $f$ be such a function. Condition 1) of $\mathrm{Im}f\in V[0]$ means we can regard $f$ as an element of $ \mathbb{C}[q^{z},q^{-z}]$ after taking into account $V[0]\cong \mathbb{C}$. So, 
$f$ is of the form $f(x)=\sum_{n}a_{n}e^{n\alpha /2}(x),$ with only finitely many $n \in \mathbb{Z}, \left| n\right| \le N$ appearing.

Condition 3) is about $E^{n}.f$ being divisible by a certain function. We are in a $U_{q}(\mathfrak{sl}_{2})$ module $V$, which has $1$ dimensional weight spaces $V[0],V[2],\ldots V[2m]$, with $E:V[2i]\to V[2i+2]$ being injective for $i=0,\ldots m-1$, and being zero for $i=m$. The functions $E^{n}.f$ are of the form rational function times a basis vector for $V[2n]$. For $n>m$ this is zero, so it is divisible by anything. For $n=1,\ldots m$, the rational function in $E^n.f$ is, up to a multiplicative constant, equal to $f$. Condition 3) in this case says that  $f$ is divisible in the ring $\mathbb{C}[e^{z},e^{-z}]$ by $$(1-q^{2+2z})(1-q^{4+2z})\ldots (1-q^{2m+2z})=$$ $$=(-1)^m \cdot q^{m(m+1)/2} \cdot q^{mz} \cdot (q^{z+1}-q^{-z-1}) (q^{z+2}-q^{-z-2}) \ldots(q^{z+m}-q^{-z-m}).$$ This is equivalent to saying it is divisible by $$(q^{z+1}-q^{-z-1}) (q^{z+2}-q^{-z-2}) \ldots(q^{z+m}-q^{-z-m}).$$

Condition 2) says that $f$ is invariant under the unshifted action of the dynamical Weyl group. As the Weyl group in this case has only one nontrivial element, call it $s$, this really means $f$ is invariant under the action of the operator $\mathcal{A}_{s,V}$, which was explicitly calculated in \ref{explicitA}. We know $s$ acts on $\mathfrak{h}$ by $-1$, so this means 
\begin{eqnarray*}
f(q^{-z})&=&\mathcal{A}_{s,V}(z)f(q^{z})\\
&=& A_{s,V}(-z-1)f(q^{z})\\
&=& (-1)^m \prod_{j=1}^{m}\frac{[-z+j]_{q}}{[-z-j]_{q}} \cdot  f(q^{z})\\
&=& \prod_{j=1}^{m}\frac{q^{j-z}-q^{-j+z}}{q^{j+z}-q^{-j-z}} \cdot  f(q^{z})
\end{eqnarray*}
Thus, the function 
$$g(q^z)= \frac{f(q^z)} {(q^{z+1}-q^{-z-1}) (q^{z+2}-q^{-z-2}) \ldots(q^{z+m}-q^{-z-m})} $$
is $W$-equivariant, in the sense $g(q^{z})=g(q^{-z})$. The function $g$ is also a Laurent polynomial, because by condition 3) above, $f$ is divisible in $\mathbb{C}[q^z,q^{-z}]$ by the denominator of $g$.  If $f$ was of degree $\le N$, then $g$ is of degree $\le N-m$. Invariance under $W$ and limitations on the maximal degree mean that $g$ is of the form $$g(q^{z})=\sum_{n=0}^{N-m}b_{n}(q^{zn}+q^{-zn}).$$
The space of all such functions has dimension $N-m+1$; so, the space of all possible $f$ that satisfy 1)-3) and have degree less or equal to $N$ also has dimension less or equal to $N-m+1$. 

This proves the lemma. 
\end{proof}

To finish the proof of Theorem \ref{proving} we need to show that any function $f\in O(H)\otimes V$ satisfying 1)-3) from the statement can be written as a linear combination of trace functions. For this,  write $f$ as a sum of characters of $H$, $$f(x)=\sum_{\mu \in P} v_{\mu}e^{\mu}(x),$$ with $v_{\mu}\in V[0]$. 

For any fixed $i=1,\ldots r$ one can decompose $f$ as follows:
$$f=\sum_{\beta\in P_{+}/\mathbb{Z}\alpha_{i}}f_{\beta}, \qquad  f_{\beta}=\sum_{\mu \in \beta}v_{\mu}e^{\mu}.$$

\begin{lemma}\label{beta} Every $f_{\beta}$ is a sum of trace functions for the subalgebra $U_{q_{i}}(\mathfrak{sl}_{2})$ of $U_{q}(\mathfrak{g})$ generated by $E_{i}, F_{i}, q^{zh_{i}}, z\in \mathbb{C}$.
\end{lemma}

\begin{proof}
Due to Lemma \ref{sl2}, it is enough to prove they satisfy 1)-3) for $U_{q_i}(\mathfrak{sl}_{2})$.

Condition 1) is clear. 

Condition 2) is about dynamical Weyl group invariance. The operators $A_{s_{i},V}(\mu):V[\nu]\to V[s_{i}\nu]$ for $U_{q}(\mathfrak{g})$ and $A_{s_{i},V}(\mu(h_i)):V[\nu]\to V[s_{i}\nu]$ for $U_{q_{i}}(\mathfrak{sl}_{2})$ coincide. Using this, the fact that $\mathcal{A}_{s_{i},V}(\lambda)=A_{s_{i},V}(-\lambda-\rho)$, and the condition that $f$ is invariant under the action of $s_{i}$, meaning $$\sum_{\beta}f_{\beta}(q^{2s_{i}\lambda})=\mathcal{A}_{s_{i},V}(\lambda)\sum_{\beta}f_{\beta}(q^{2\lambda}),$$ we get  $$\sum_{\beta}f_{\beta}(q^{2s_{i}\lambda})=\sum_{\beta}A_{s_{i},V'}(-d_{i}^{-1}(\left<\alpha_{i}, \lambda \right>+1)) f_{\beta}(q^{2\lambda}).$$
Now decompose both of these functions into their $\beta$ parts as we did with $f$. Call the left hand side function $l$ and the right hand side $r$. Using $\left< \mu, s_{i}\lambda \right>=\left< s_{i}\mu, \lambda \right>$ and the fact that $\mu \in \beta$ implies $s_{i}\mu \in \beta$, we get $l_{\beta}(q^{\lambda})=f_{\beta}(q^{2s_{i} \lambda})$. On the right hand side, we know that $A_{s_i,V}(-d_{i}^{-1}(\left< \alpha_{i}, \lambda \right>+1))$ maps $V[\nu]$ to $V[s_{i}\nu]$, so $r_{\beta}(q^{\lambda})=A_{s_i,V}(-d_{i}^{-1}(\left< \alpha_{i}, \lambda \right>+1)) f_{\beta}(q^{2\lambda})$. Thus we have
$$f_{\beta}(q^{2s_{i}\lambda})=A_{s_{i},V}(-d_{i}^{-1}(\left< \alpha_{i}, \lambda \right>+1)) f_{\beta}(q^{2\lambda}).$$ Remembering the identification $\mathfrak{h}\cong \mathfrak{h}^*$  and restricting this to $2\lambda=zd_{i}^{-1}\alpha_{i}$, which corresponds to $zh_{i}$ we get
$$f_{\beta}(q^{-zh_{i}})=A_{s_{i},V'}(d_{i}^{-1}(-z-1)) f_{\beta}(e^{zh_{i}}).$$
This is exactly the dynamical Weyl group invariance for $U_{q}(\mathfrak{sl}_{2})$, as the function $f$ is defined in terms of powers of $q$, and the operator $A_{s_{i},V}$, which was defined for $U_{q_{i}}(\mathfrak{sl}_{2})$, in terms of $q_{i}=q^{d_{i}}$. Replacing $q$ with $q_{i}$ we get the required dynamical Weyl group invariance for $U_{q_{i}}(\mathfrak{sl}_{2})$.

Condition 3): we know that $E_{i}^n.f$ is divisible by $(1-q_{i}^2e^{\alpha_{i}})\ldots (1-q_{i}^{2n}e^{\alpha_{i}})$, call the quotient $g\in O(H)$ and write $$E_{i}^n.\sum_{\beta}f_{\beta}=(1-q_{i}^2e^{\alpha_{i}})\ldots (1-q_{i}^{2n}e^{\alpha_{i}})\cdot \sum_{\beta}g_{\beta}.$$ Decompose both sides into their $\beta$ parts to get $$E_{i}^n.f_{\beta}=(1-q_{i}^2e^{\alpha_{i}})\ldots (1-q_{i}^{2n}e^{\alpha_{i}})g_{\beta}.$$ Replacing $q$ by $q_{i}$ in all the functions we get the required statement for $U_{q_{i}}(\mathfrak{sl}_{2})$.

\end{proof}

For $f$ as above, $f=\sum_{\mu \in P} v_{\mu}e^{\mu}$, let $D(f)= \{ \mu\in P | v_{\mu}\ne 0 \}$, and let $C(f)$ be the convex hull of $D(f)$ in the Euclidean space $\mathfrak{h}_{\mathbb{R}}^*$. Then define the weight diagram of $f$ to be the set $\bold{WD}(f)=C(f)\cap P$. We will prove the theorem by induction on the size of the set $\bold{WD}(f)$.

\begin{lemma}
If $f$ satisfies 1)-3), then $D(f)$, and consequently $\bold{WD}(f)$, is $W$-invariant.
\end{lemma}
\begin{proof}
First note that it follows directly from the proof of Lemma \ref{sl2} that this is true for $U_{q}(\mathfrak{sl}_{2})$, as $$f(q^{zh})=g(q^{zh})\cdot (q^{z+1}-q^{-z-1})\ldots (q^{z+m}-q^{-z-m})$$ with $g$ being $W$-equivariant, so the set of powers of $q^z$ that appear in $f$ is symmetric around $0$. This means $D(f)$, and therefore also $C(f)$ and $\bold{WD}(f)$, is $W$-invariant.

Next, write $f=\sum f_{\beta}$ as before, for $\beta \in P_{+}/\mathbb{Z}\alpha_{i}$. The simple reflection $s_{i}$ preserves the equivalence classes $\beta$. So, $s_{i} D(f)=D(f)$ if and only if $s_{i}D(f_{\beta})=D(f_{\beta})$.

From the previous lemma, $f_{\beta}$ is the sum of trace functions, so from the comment at the beginning of this proof, $D(f_{\beta})$ is invariant under the action of the Weyl group of  $U_{q_{i}}(\mathfrak{sl}_{2})$. The nontrivial element of this Weyl group is the reflection with respect to the only simple root for $U_{q_{i}}(\mathfrak{sl}_{2})$, which is $\alpha_{i}$. So, every $D(f_{\beta})\subseteq P$ is preserved under $s_{i}$, hence $s_{i}D(f)=D(f)$ and $s_{i}\bold{WD}(f)=\bold{WD}(f)$.

Geometrically in the lattice $P$, the argument in the last paragraph corresponds to decomposing $D(f)$ into sets $D(f_{\beta})$, so that every $D(f_{\beta})$ consists of points of $D(f)$ that lie on one of the parallel lines in $\mathfrak{h}^*$, passing through $\beta$, in the direction of $\alpha_{i}$. Then we note that $s_{i}$ preserves each of these lines, and that every such line is symmetric with respect to the hyperplane through $0$ orthogonal to $\alpha_{i}$. This is exactly the reflection hyperplane of $s_{i}$, so $D(f)$ is symmetric with respect to this hyperplane and preserved by $s_{i}$.

Of course, once we proved $D(f)$ and $\bold{WD}(f)$ are preserved by all the simple reflections $s_i$, we immediately conclude that they are preserved by the entire group $W$ generated by all the $s_{i}$.

\end{proof}

\begin{proof}[Proof of Theorem \ref{proving}]
Let us prove Theorem \ref{proving} by induction on the size of the finite set $\bold{WD}(f)$.

If the set $\bold{WD}(f)$ is empty, $f=0$ and there is nothing to prove.

Otherwise, assume we have proved the theorem for all functions whose weight diagram has fewer elements than $\bold{WD}(f)$.

Pick $\mu \in \bold{WD}(f)$ an extremal point, meaning a point $\mu$ such that it is not in the convex hull of $\bold{WD}(f)\backslash \{\mu\}$. Such a point exists, as $\bold{WD}(f)$ is a finite set. Moreover, such a point $\mu$ is in $D(f)$. To see that, notice that $\mu \in C(f)$ means that either $\mu \in D(f)$, or $\mu=\sum_{i}t_{i}\mu_{i}$ for some $t_{i}\ge 0, \sum_{i}t_{i}=1$ and some $\mu_{i}\in D(f), \mu_{i}\ne \mu$. In the latter case $\mu_{i}\in D(f)\backslash \{\mu\} \subseteq \bold{WD}(f)\backslash \{\mu\}$, so $\mu =\sum_{i}t_{i}\mu_{i}$ is in the convex hull of $ \bold{WD}(f)\backslash \{\mu\}$, contrary to the choice of $\mu$.

As $D(f)$ and $\bold{WD}(f)$ are $W$-invariant, we can assume without loss of generality that $\mu$ is a dominant weight. Finally, for any $i=1,\ldots r$, the weight $\mu + \alpha_{i}$ is not in $\bold{WD}(f)$. To see that, consider two cases: either $\left< \mu, \alpha_{i}\right>\ne 0$ or $\left< \mu, \alpha_{i}\right>=0$. If $\left< \mu, \alpha_{i}\right> \ne 0$, then $s_{i}\mu=\mu-2\frac{\left< \mu, \alpha_{i}\right>}{\left< \alpha_{i}, \alpha_{i}\right>}\alpha_{i} \in \bold{WD}(f) \backslash \{\mu\}, \mu + \alpha_{i}\in \bold{WD}(f)\backslash \{\mu\}$ implies that $\mu$ is in the convex hull of $\bold{WD}(f)\backslash \{\mu\}$, contrary to the choice of $\mu$. If $\left< \mu, \alpha_{i}\right> = 0$, then the same can be concluded from $\mu + \alpha_{i}\in \bold{WD}(f)\backslash \{\mu\}, s_{i}(\mu + \alpha_{i})=\mu - \alpha_{i}\in \bold{WD}(f)\backslash \{\mu\}$.

Let us now restrict $f$ to $U_{q_{i}}(\mathfrak{sl}_{2})$ and decompose into $f_{\beta}$ as before. Then Lemma \ref{beta} tells us that all $f_{\beta}$, and in particular the $f_{\beta}$ such that $\mu\in \beta$, are traces of intertwining operators for $U_{q_{i}}(\mathfrak{sl}_{2})$. Lemma \ref{iso}, 2) then implies that $E_{i}^{d_{i}^{-1}\left<\mu,\alpha_{i} \right>+1}v_{\mu}=0$.

Since this statement is valid for every $i=1\ldots r$, the same Lemma \ref{iso}, 2) implies that there is an intertwining operator $\bar{\Phi}^{v_{\mu}}_{\mu}: L_{\mu}\to L_{\mu}\otimes V$. Its trace function $\Psi ^{v_{\mu}}_{\mu}$ has $\bold{WD}(\Psi ^{v_{\mu}}_{\mu})$ equal to the convex hull of the set of weights of $L_{\mu}$, i.e. equal to the convex hull of the $W$ orbit of $\mu$.  This is contained in the set $\bold{WD}(f)$. So, the function $f-\Psi^{v_{\mu}}_{\mu}$ satisfies 1)-3), has $\bold{WD}(f-\Psi^{v_{\mu}}_{\mu})$ contained in $\bold{WD}(f)$, and has the coefficient of $e^{\mu}$ equal to $v_{\mu}-v_{\mu}=0$. This means $D(f-\Psi^{v_{\mu}}_{\mu})$ is a subset of $\bold{WD}(f)$ which does not contain $\mu$, so it's convex hull doesn't contain $\mu$, and so  $\bold{WD}(f-\Psi^{v_{\mu}}_{\mu})$ is a proper subset of $\bold{WD}(f)$. 

By induction assumption we can now express $f-\Psi^{v_{\mu}}_{\mu}$ as a linear combination of trace functions. So, we can express $f$ as a linear combination of trace functions. This proves the theorem.

\end{proof}

\begin{remark} \label{last}
It is explained in \cite{KNV} how theorem \ref{KNV1} reduces when $V$ is small enough in the appropriate sense. First, if $V$ is a trivial representation of $\mathfrak{g}$, then conditions $(1)$ and $(3)$ of theorem \ref{KNV1} are automatically satisfied, so the statement becomes the Chevalley isomorphism theorem $\mathbb{C}[\mathfrak{g}]^G \tilde{\to} \mathbb{C}[\mathfrak{h}]^W$. The second special case is when $V$ is \textit{small} in the sense of \cite{Broer}, meaning that for every root $\alpha$, $2\alpha$ is not a weight of $V$. In that case, any function $f\in \mathbb{C}[\mathfrak{h}]\otimes V$ that satisfies conditions $(1)$ and $(2)$ automatically satisfies $(3)$ as well. To see that, first note that $E_{i}^{n}.f=0$ for all $n\ge 2$. Next, for any vector $v\in V[0]$, either $E_{i}.v=0$, or $E_{i}.v\ne 0, E_{i}^2.v=0$ and $v$ generates an $(\mathfrak{sl}_{2})_{i}$ representation isomorphic to either the three dimensional irreducible representation $L_{2}$, or the direct sum of the trivial representation $L_{0}$ with $L_{2}$. Then condition $(1)$ means that $f$ is a sum of functions of the form $f_{1}v_{1}$ and functions of the form $f_{2}v_{2}$, for some $f_{1,2}\in \mathbb{C}[\mathfrak{h}]$, some $v_{1}$ which generate a trivial $(\mathfrak{sl}_{2})_{i}$ representation and some $v_{2}$ in the zero weight space of some three dimensional $(\mathfrak{sl}_{2})_{i}$ representation. Condition $(2)$ implies that $f_{1}$ is an even function and $f_{2}$ an odd one with respect to the action of the element $s_{i}\in W$ corresponding to $(\mathfrak{sl}_{2})_{i}$, so $E_{i}.f(h)=f_{2}(h)E_{i}.v_{2}$ is divisible by $\alpha_{i}$. This reduces the theorem \ref{KNV1} to theorem 1 in \cite{Broer}.

In the context of quantum groups, the same analysis applies to theorem \ref{main}. If $V$ is trivial, then conditions $(1)$ and $(3)$ are satisfied, and condition $(2)$ reduces to $f(q^{w\lambda})=f(q^{\lambda})$ because $\mathcal{A}_{w,V}(\lambda)=id$ (by lemma \ref{explicitA}). If $V$ is small in the sense of \cite{Broer}, then for any copy of $U_{q}(\mathfrak{sl}_{2})_{i}$, the only representations of it that contain a nonzero vector in $V[0]$ are direct sums of trivial and three dimensional irreducible representations. We again conclude that any function $f\in O(H)\otimes V$ which satisfies $(1)$ and $(2)$ must be a sum of functions of the form $f_{1}v_{1}$ and  $f_{2}v_{2}$, with $v_{1}$ in some copy of $L_{0}$ and $v_{2}$ in some copy of $L_{2}$. If $f$ also satisfies $(2)$, then by the proof of lemma \ref{sl2}, $f_{2}$ is of the form $f_{2}(q^z)=g(q^z)\cdot (q^{z+1}-q^{-z-1})$, so $E_{i}.f_{2}$ is divisible by $(1-q_{i}^2e^{\alpha_{i}})$. All the other parts of condition $(3)$ are satisfied trivially, as $E_{i}.v_{1}=0$ and $E_{i}^2.v_{2}=0$. So, in the case $V$ is small, condition $(3)$ is unnecessary, and theorem \ref{main} reduces to a quantum version of theorem 1 from \cite{Broer}.

\end{remark}


\begin{thebibliography}{99}
\begin{normalsize}

\bibitem[Br]{Broer} Broer, A, \emph{The sum of generalized exponents and Chevalley's restriction theorem for modules of coinvariants}, Indag. Math, New Ser. 6, No 4, 385-396 (1995)

\bibitem[Bo]{Borel} Borel, A. \emph{Linear algebraic groups}, Springer-Verlag, New York (1991)

\bibitem[EV1]{EV1} Etingof, P. and Varchenko, A, \emph{Dynamical Weyl Groups and Applications}, Advances in Mathematics 167, 74-127 (2002)

\bibitem[EV2]{EV2} Etingof, P. and Varchenko, A, \emph{Traces of Intertwiners for Quantum Groups and Difference Equations, I}, Duke Mathematical Journal Vol. 104 No. 3, 391-432 (2000)

\bibitem[EFK]{EFK}Etingof, P, Frenkel, I, Kirillov, A, Jr, \emph{Spherical Functions on Affine Lie Groups}, Duke Mathematical Journal Vol. 80 No. 1, 59-90 (1995)

\bibitem[KNV]{KNV} Khoroshkin, S, Nazarov, M. and Vinberg, E, \emph{A Generalized Harish-Chandra Isomorphism}, Adv. Math. 226, 1168-1180 (2011)

\bibitem[KS]{KS} Korogodski, L. and Soibelman, Y, \emph{Algebras of Functions on Quantum Groups: Part I}, Mathematical Surveys and Monographs, Vol 56, AMS (1998)

\bibitem[T]{T} Tanisaki, T, \emph{Harish-Chandra Isomorphisms for Quantum Algebras}, Commun. Math. Phys. 127, 555-571 (1990)


\end{normalsize}
\end{thebibliography}
\end{document}